\pdfoutput=1
\RequirePackage{ifpdf}
\ifpdf % We~are running pdfTeX in pdf mode
\documentclass[pdftex]{sigma}%,draft
\else
\documentclass{sigma}
\fi

\numberwithin{equation}{section}

\newtheorem{Theorem}{Theorem}[section]
\newtheorem{Corollary}[Theorem]{Corollary}
\newtheorem{Lemma}[Theorem]{Lemma}
\newtheorem{Proposition}[Theorem]{Proposition}
 { \theoremstyle{definition}

\newtheorem{Remark}[Theorem]{Remark} }

\newcommand{\la}{\langle}
\newcommand{\ra}{\rangle}
\newcommand\qbin[3]{\left[\begin{matrix} #1 \\ #2 \end{matrix} \right]_{#3}}
\def\tinyA{{\mbox{\tiny $A$}}}
\def\tinyR{{\mbox{\tiny $R$}}}

\begin{document}
\allowdisplaybreaks

\newcommand{\arXivNumber}{1910.08393}

\renewcommand{\thefootnote}{}

\renewcommand{\PaperNumber}{113}

\FirstPageHeading

\ShortArticleName{$q$-Difference Systems for the Jackson Integral of Symmetric Selberg Type}

\ArticleName{$\boldsymbol{q}$-Difference Systems for the Jackson Integral\\ of Symmetric Selberg Type\footnote{This paper is a~contribution to the Special Issue on Elliptic Integrable Systems, Special Functions and Quantum Field Theory. The full collection is available at \href{https://www.emis.de/journals/SIGMA/elliptic-integrable-systems.html}{https://www.emis.de/journals/SIGMA/elliptic-integrable-systems.html}}}

\Author{Masahiko ITO}

\AuthorNameForHeading{M.~Ito}

\Address{Department of Mathematical Sciences, University of the Ryukyus, Okinawa 903-0213, Japan}
\Email{\href{mailto:mito@sci.u-ryukyu.ac.jp}{mito@sci.u-ryukyu.ac.jp}}

\ArticleDates{Received April 29, 2020, in final form October 29, 2020; Published online November 08, 2020}

\Abstract{We provide an explicit expression for the first order $q$-difference system for the Jackson integral of symmetric Selberg type. The $q$-difference system gives a generalization of $q$-analog of contiguous relations for the Gauss hypergeometric function. As a basis of the system we use a set of the symmetric polynomials introduced by Matsuo in his study of the $q$-KZ equation. Our main result is an explicit expression for the coefficient matrix of the $q$-difference system in terms of its Gauss matrix decomposition. We introduce a~class of symmetric polynomials called {\it interpolation polynomials}, which includes Matsuo's polynomials. By repeated use of three-term relations among the interpolation polynomials we compute the coefficient matrix.}


\Classification{33D60; 39A13}

\renewcommand{\thefootnote}{\arabic{footnote}}
\setcounter{footnote}{0}

\section{Introduction}
The Gauss hypergeometric function
\begin{equation*}
_2F_1
\left(
\begin{matrix}
 a,b \\
 c
\end{matrix}
;x
\right)
=\frac{\Gamma(c)}{\Gamma(a)\Gamma(c-a)}\int_0^1
z^{a-1}(1-z)^{c-a-1}(1-xz)^{-b}\,{\rm d}z,
\end{equation*}
where $\operatorname{Re}c>\operatorname{Re}a>0$ and $|x|<1$,
satisfies the contiguous relations
\begin{equation}
\label{eq:conti1}
_2F_1
\left(
\begin{matrix}
 a,b \\
 c
\end{matrix}
;x
\right)
={}_2F_1
\left(
\begin{matrix}
 a,b+1 \\
 c+1
\end{matrix}
;x
\right)
-x\frac{a(c-b)}{c(c+1)}
{}_2F_1
\left(
\begin{matrix}
 a+1,b+1 \\
 c+2
\end{matrix}
;x
\right)
\end{equation}
and
\begin{equation}
\label{eq:conti2}
_2F_1
\left(
\begin{matrix}
 a,b \\
 c
\end{matrix}
;x
\right)
={}_2F_1
\left(
\begin{matrix}
 a+1,b \\
 c+1
\end{matrix}
;x
\right)
-x\frac{b(c-a)}{c(c+1)}
{}_2F_1
\left(
\begin{matrix}
 a+1,b+1 \\
 c+2
\end{matrix}
;x
\right).
\end{equation}
These contiguous relations for the Gauss hypergeometric function are extended to
a difference system for a function defined by multivariable integral
with respect to the Selberg type kernel~\cite{V2}
\begin{equation*}
\Psi(z):=\prod_{i=1}^n z_i^{\alpha-1}(1-z_i)^{\beta-1}(x-z_i)^{\gamma-1}
\prod_{1\le j< k\le n}|z_j-z_k|^{2\tau}.
\end{equation*}
For the integral
\begin{equation*}
\la e_i \ra:=\int_C e_i(z)\Psi(z)\,{\rm d}z_1\cdots {\rm d}z_n, \qquad i=0,1,\ldots,n,
\end{equation*}
where $e_i(z)$ is the function specified by
\begin{equation*}
e_i(z):=\prod_{j=1}^{n-i}(x-z_j)\prod_{k=n-i+1}^n(1-z_k)
\end{equation*}
and $C$ is some suitable region, the $(n+1)$-tuple $(\la e_0\ra,\la e_1\ra,\ldots,\la e_n\ra)$
satisfies the following difference system.
Let $\delta_{ij}$ be the symbol of Kronecker's delta.
\begin{Proposition}[{\cite[Theorem~2.2]{FI10}}]
\label{prop:FI10}
Let $T_\alpha$ be the shift operator with respect to $\alpha\to \alpha+1$,
i.e., $T_\alpha f(\alpha)=f(\alpha+1)$
for an arbitrary function $f\colon \mathbb{C}\to \mathbb{C}$.
Then
\begin{equation}
\label{eq:system1}
T_\alpha(\la e_0\ra,\la e_1\ra,\ldots,\la e_n\ra)=(\la e_0\ra,\la e_1\ra,\ldots,\la e_n\ra)M,
\end{equation}
where the $(n+1)\times(n+1)$ matrix $M$ is written in terms of its Gauss matrix decomposition as
\begin{equation}
\label{eq:system1-1}
M=LDU=U'D'L'.
\end{equation}
Here $L=(l_{ij})_{0\le i,j\le n}$, $D=(d_{j}\delta_{ij})_{0\le i,j\le n}$, $U=(u_{ij})_{0\le i,j\le n}$ are
the lower triangular, diagonal, upper triangular matrices, respectively, given by
\begin{gather*}
l_{ij}=(-x)^{i-j} {n-j\choose n-i}
\frac{(\gamma+j\tau;\tau)_{i-j}}
{(\alpha+\gamma+ 2j\tau;\tau)_{i-j}}
,\\
d_{j}=
\frac{x^j
(\alpha;\tau)_{j}(\alpha+\gamma+2j\tau;\tau)_{n-j}}
{(\alpha+\gamma+(j-1)\tau;\tau)_{j}(\alpha+\beta+\gamma+(n+j-1)\tau;\tau)_{n-j}},\\
u_{ij}=(-1)^{j-i} {j\choose i}
\frac{(\beta+(n-j)\tau;\tau)_{j-i}}
{(\alpha+\gamma+ 2i\tau;\tau)_{j-i}},
\end{gather*}
and $U'=(u'_{ij})_{0\le i,j\le n}$, $D'=(d'_{j}\delta_{ij})_{0\le i,j\le n}$, $L'=(l'_{ij})_{0\le i,j\le n}$
are the upper triangular, diagonal, lower triangular matrices,
respectively, given by
\begin{gather*}
u'_{ij}=(-x^{-1})^{j-i}{j\choose i}
\frac{(\beta+(n-j)\tau;\tau)_{j-i}
}{(\alpha+\beta+2(n-j)\tau;\tau)_{j-i}}
,\\
d'_{j}=\frac{x^j(\alpha+\beta+2(n-j)\tau;\tau)_j(\alpha;\tau)_{n-j}
}{(\alpha+\beta+\gamma+(2n-j-1)\tau;\tau)_j(\alpha+\beta+(n-j-1)\tau;\tau)_{n-j}}
,\\
l'_{ij}=(-1)^{i-j}{n-j\choose n-i}\frac{(\gamma +j\tau;\tau)_{i-j}}
{(\alpha+\beta+2(n-i)\tau;\tau)_{i-j}}
,\end{gather*}
where $(x;\tau)_0:=1$ and $(x;\tau)_i:=x(x+\tau)(x+2\tau)\cdots (x+(i-1)\tau)$ for $i=1,2,\ldots$.
\end{Proposition}

In particular, when $n=1$ the system (\ref{eq:system1}) is given by
\begin{subequations}
\begin{gather}
T_\alpha(\la e_0\ra,\la e_1\ra)
 =(\la e_0\ra,\la e_1\ra)
\left(\!\!
\begin{array}{cc}
\alpha+\gamma & 0 \\
-x\gamma & x\alpha
\end{array}
\!\right)
\begin{pmatrix}
\alpha+\beta+\gamma & \beta \\
0 & \alpha+\gamma
\end{pmatrix}^{-1},
\label{eq:system1-2a}\\
T_\alpha(\la e_0\ra,\la e_1\ra)
=(\la e_0\ra,\la e_1\ra)
\begin{pmatrix}
\alpha &-\beta \\
0 & x(\alpha+\beta)
\end{pmatrix}
\begin{pmatrix}
\alpha+\beta & 0 \\
\gamma & \alpha+\beta+\gamma
\end{pmatrix}^{-1}.
\label{eq:system1-2b}
\end{gather}
\end{subequations}
The system (\ref{eq:system1-2b}) can be rewritten as the following
system of three-term equations
\begin{equation}\label{eq:conti3}
(\alpha+\beta+\gamma)T_\alpha\la e_{1}\ra=-\beta \la e_{0}\ra+x(\alpha+\beta)\la e_{1}\ra,
\end{equation}
and
\begin{equation}\label{eq:conti4}
\gamma T_\alpha \la e_{1}\ra+(\alpha+\beta)T_\alpha \la e_{0}\ra=\alpha \la e_{0}\ra.
\end{equation}
Since, for $n=1$
\begin{equation*}
\la e_0\ra=\int_0^1 z^{\alpha-1}(1-z)^{\beta-1}(x-z)^{\gamma}\,{\rm d}z \qquad\mbox{and}\qquad
\la e_1\ra=\int_0^1 z^{\alpha-1}(1-z)^{\beta}(x-z)^{\gamma-1}\,{\rm d}z,
\end{equation*}
under the conditions $\operatorname{Re}\alpha>0$, $\operatorname{Re}\beta>0$ and $|x|>1$
we see that the equations (\ref{eq:conti3}) and (\ref{eq:conti4}) exactly coincide with the
contiguous relations~(\ref{eq:conti1}) and (\ref{eq:conti2}), respectively,
after the substitutions $\alpha \to a$, $\beta\to c-a$, $\gamma\to -b$ and $x\to 1/x$.
Therefore the difference system~(\ref{eq:system1}) expressed in terms of Gauss matrix decomposition (\ref{eq:system1-1})
can be regarded as a natural extension of the
contiguous relations~(\ref{eq:conti1}) and~(\ref{eq:conti2}).
For further applications of the difference system~(\ref{eq:system1}) in random matrix theory, see~\cite{FI10}.
In the discussion of the result being a generalization of contiguous relations of the Gauss hypergeometric function,
the analog of the equation~\eqref{eq:conti1} for the Selberg integral can be found in~\cite{FW08, Kan93}.

Next we would like to discuss a $q$-analogue of the difference system~(\ref{eq:system1}) in Proposition~\ref{prop:FI10}.
This is one of the aims of this paper.
For an arbitrary $c\in \mathbb{C}^*$ we use the {\em $c$-shifted factorial} for $x\in \mathbb{C}$
\begin{equation*}
(x;c)_i=
\begin{cases}
(1-x)(1-cx)\cdots\big(1-c^{i-1}x\big)
 & \mbox{if}\quad i=1,2,\ldots , \\
1 & \mbox{if}\quad i=0, \\
\displaystyle \frac{1}{\big(1-c^{-1}x\big)\big(1-c^{-2}x\big)\cdots\big(1-c^{i}x\big)} &
\mbox{if}\quad i=-1,-2,\ldots,
\end{cases}
\end{equation*}
and the {\em $c$-binomial coefficient}
\begin{equation*}
\qbin{i}{j}{c}=\frac{(c;c)_i}{(c;c)_{i-j}(c;c)_j}.
\end{equation*}
We also use the symbol
$(x;c)_\infty:=\prod_{i=0}^\infty\big(1-c^{i}x\big)$ for $|c|<1$.
Throughout this paper we fix $q\in\mathbb{C}^*$ with $|q|<1$.
For a point $\xi=(\xi_1,\ldots,\xi_n)\in (\mathbb{C}^*)^n$ and a function $f(z)=f(z_1,\ldots,z_n)$ on~$(\mathbb{C}^*)^n$
we define the following sum over the lattice $\mathbb{Z}^n$ by
\begin{equation}\label{eq:00jac4}
\int_0^{\xi\infty}f(z) \frac{{\rm d}_qz_1}{z_1}\wedge\cdots\wedge\frac{{\rm d}_qz_n}{z_n}
:=(1-q)^n\sum_{(\nu_1,\ldots,\nu_n)\in \mathbb{Z}^n}f\big(\xi_1 q^{\nu_1},\ldots,\xi_n q^{\nu_n}\big),
\end{equation}
if it converges. We call it the {\it Jackson integral of $f(z)$}. By definition the Jackson integral (\ref{eq:00jac4})
is invariant under the $q$-shift $\xi_i\to q\xi_i$ ($i=1,\ldots,n$).
Let $\Phi_{n,m}(z)$ and $\Delta(z)$ be the functions on~$(\mathbb{C}^*)^n$ specified by
\begin{gather}\label{eq:Phi}
\Phi_{n,m}(z)
:=\prod_{i=1}^n
\left\{z_i^\alpha
\prod_{r=1}^m
\frac{\big(qa_r^{-1}z_i;q\big)_\infty}
{(b_rz_i;q)_\infty}\right\}
\prod_{1\le j<k\le n}z_j^{2\tau-1}\frac{\big(qt^{-1} z_k/z_j;q\big)_\infty}{(t z_k/z_j;q)_\infty},
\\ %\label{eq:Delta}
\Delta(z):=\prod_{1\le i<j\le n}(z_i-z_j),\nonumber
\end{gather}
where $t=q^\tau$.
For a point $\xi=(\xi_1,\ldots,\xi_n)\in (\mathbb {C}^*)^n$ and an arbitrary symmetric function $\phi(z)=\phi(z_1,\ldots,z_n)$ on $(\mathbb {C}^*)^n$ we set
\begin{equation*}
\la \phi,\xi\ra:=\int_{0}^{\xi\infty}\phi(z)\Phi_{n,m}(z)\Delta(z)
\frac{{\rm d}_qz_1}{z_1}\wedge\cdots\wedge\frac{{\rm d}_qz_n}{z_n},
\end{equation*}
which we call the {\it Jackson integral of symmetric Selberg type}.
In the study of $q$-difference de Rham cohomology associated with Jackson integrals~\cite{Ao90,AK91},
Aomoto and Kato~\cite{AK93-1} showed that the Jackson integral of symmetric Selberg type
satisfies $q$-difference systems of rank ${n+m-1\choose m-1}$ when the parameters are generic.
When $m=1$ the Jackson integral of symmetric Selberg type is equivalent to
the $q$-Selberg integral defined by Askey~\cite{As80} and proved by others, see~\cite{Ao98,Ev92,Ha88, Kad88} for instance.
See also recent references \cite[Section~2.3]{FW08} and~\cite{IF2017}.
$q$-Selberg integral is a very active area of research with important
connections to special functions, combinatorics, mathematical physics and orthogonal polynomials
(see \cite{ARW, KO, KS, KY,RTVZ,W2005} and \cite[Section~5]{IZ}).
Using the Jackson integral of symmetric Selberg type for $m=2$, Matsuo \cite{ma1,ma2} constructed a set of solutions of the $q$-KZ equation.
Varchenko~\cite{V1} extended Matsuo's construction to more general setting of the $q$-KZ equation using
the Jackson integral of symmetric Selberg type for general $m$.
Writing $a_r=x_r$, $b_r=q^{\beta_r}x_r^{-1}$ in \eqref{eq:Phi}, the $q$-KZ equation they studied can be regarded as the $q$-difference system with respect to the $q$-shift $x_r\to qx_r$ $(r=1,\ldots,m)$.
In another context, writing $a_r=qx_r^{-1}$, $b_r=q^{\mu_r}x_r$ in \eqref{eq:Phi},
Kaneko \cite{Kan96} showed an explicit expression for the
$q$-difference system with respect to the $q$-shift $x_r\to qx_r$ $(r=2,\ldots,m)$ satisfied by
the Jackson integral of symmetric Selberg type for general $m$ with special constraints
$\mu_2=\cdots=\mu_m=1$ or $\mu_2=\cdots=\mu_m=-\tau$. With these constraints
the $q$-difference system degenerates to be very simple and it can also be regarded as a generalization of
the second order $q$-difference equation satisfied by Heine's $_2\phi_1$ $q$-hypergeometric function.

In this paper, we fix $m=2$ for \eqref{eq:Phi}, and study two types of $q$-difference systems for the Jackson integral of symmetric Selberg type for $\Phi(z)=\Phi_{n,2}(z)$.
One is the $q$-difference system with respect to the shift $\alpha\to \alpha +1$, and
the other is the system with respect to the $q$-shifts $a_i \to qa_i$ and $b_i \to q^{-1}b_i$ simultaneously.
For these purposes, we define the set of symmetric
polynomials $\{e_i(a,b;z)\,|\, i=0,1,\ldots,n\}$, where
\begin{equation}\label{eq:matsuo}
e_i(a,b;z):=\frac{1}{\Delta(z)}\times
{\cal A}\left(
\prod_{j=1}^{n-i}(1-bz_j)\prod_{j=n-i+1}^n\big(1-a^{-1}z_j\big)
\prod_{1\le k<l\le n}\big(z_k-t^{-1}z_l\big)
\right),
\end{equation}
which we call {\it Matsuo's polynomials}.
The symbol $\cal A$ means the skew-symmetrization
(see the definition~\eqref{eq:00Af} of $\cal A$ in Section~\ref{section02}).
With these symmetric polynomials, we denote
\begin{equation}\label{eq:<e_i(a,b),x>}
\la e_i(a,b),\xi\ra
:=\int_0^{\xi\infty} e_i(a,b;z)\Phi(z)\Delta(z)
\frac{{\rm d}_qz_1}{z_1}\wedge\cdots\wedge\frac{{\rm d}_qz_n}{z_n}.
\end{equation}
We assume that
\begin{equation*}
\big|qa_1^{-1}a_2^{-1}b_1^{-1}b_2^{-1} \big|<\big|q^\alpha\big|<1\qquad\mbox{and}\qquad
\big|qa_1^{-1}a_2^{-1}b_1^{-1}b_2^{-1} \big|<\big|q^\alpha t^{2n-2}\big|<1
\end{equation*}
for convergence of the Jackson integrals \eqref{eq:<e_i(a,b),x>}.
(See \cite[Lemma~3.1]{IN2018} for details of convergence.)

For the polynomials \eqref{eq:matsuo} let $R$ be the $(n+1)\times(n+1)$ matrix defined by
\begin{gather}
\big(e_n(a_2,b_1;z),e_{n-1}(a_2,b_1;z),\ldots,e_0(a_2,b_1;z)\big)\nonumber\\
\qquad {}=\big(e_0(a_1,b_2;z),e_1(a_1,b_2;z),\ldots,e_n(a_1,b_2;z)\big)R.\label{eq:R}
\end{gather}
The transition matrix $R$ is called the {\it $R$-matrix} in the context of~\cite{ma2}.
Matsuo~\cite{ma2} gave the $q$-difference system with respect to the $q$-shifts $a_i \to qa_i$ and $b_i \to q^{-1}b_i$
simultaneously,
using Matsuo's polynomials as follows.

\begin{Proposition}[Matsuo]\label{prop:Matsuo}
Let $T_{q,u}$ be the $q$-shift operator with respect to $u\to qu$, and
$T_{q,b_i}^{-1}T_{q,a_i}$ $(i=1,2)$ denote the $q$-shift operator with respect to $a_i\to qa_i$ and $b_i\to q^{-1}b_i$ simultaneously.
Then, the Jackson integrals of symmetric Selberg type satisfy the $q$-difference system with respect to $T_{q,b_i}^{-1}T_{q,a_i}$ $(i=1,2)$ given by
\begin{gather}
T_{q,b_i}^{-1}T_{q,a_i}\big(\la e_n(a_2,b_1),\xi\ra,\la e_{n-1}(a_2,b_1),\xi\ra,\ldots,\la e_0(a_2,b_1),\xi\ra\big)\nonumber\\
\qquad=\big(\la e_n(a_2,b_1),\xi\ra,\la e_{n-1}(a_2,b_1),\xi\ra,\ldots,\la e_0(a_2,b_1),\xi\ra\big)
K_i,\label{eq:Ki}
\end{gather}
whose coefficient matrices $K_i$ are expressed as $K_1=R^{-1}D_1$ and $K_2=D_2\big(T_{q,b_2}^{-1}T_{q,a_2}R\big)$,
where~$R$ is the $(n+1)\times(n+1)$ matrix given by~\eqref{eq:R}, and $D_1$, $D_2$ are the diagonal matrices given by
\begin{equation*}
D_1=\big(\big(q^\alpha t^{n-1}\big)^{n-i}\delta_{ij}\big)_{0\le i,j\le n},\qquad
D_2=\big(\big(q^\alpha t^{n-1}\big)^{i} \delta_{ij}\big)_{0\le i,j\le n}.
\end{equation*}
\end{Proposition}

\begin{Remark}
If we replace $a_i$ and $b_i$ as $a_i=x_i$ and $b_i=q^{\beta_i}x_i^{-1}$, respectively,
then \eqref{eq:Ki} simplifies to the case considered by Matsuo, and then
$T_{q,b_i}^{-1}T_{q,a_i}$ in \eqref{eq:Ki} becomes the single $q$-shift operator $T_{q,x_i}$,
and the system \eqref{eq:Ki} coincides with the $q$-KZ equation (see \cite{ma1,ma2,V1}).
For the problem of finding the explicit form of the coefficient matrix $K_i$
for the system \eqref{eq:Ki}
Aomoto and Kato used the information of a connection matrix \cite{AK95} between two kinds of fundamental solutions of \eqref{eq:Ki}
specified by their asymptotic behaviors. Based on Birkhoff's classical theory they introduced a way to
derive the explicit form of the coefficient matrix for a linear ordinary $q$-difference system from its connection matrix.
They call their method the {\it Riemann--Hilbert approach for $q$-difference equations from connection
matrices} \cite{Ao95}, and they presented $K_i$ explicitly when $n=1$ and $2$ as an example of their method (see \cite[p.~272, examples]{AK98}). The problem of finding the explicit form of the coefficient matrix
for the equivalent $q$-difference system with respect to $T_{q,x_i}$
was also studied by Mimachi~\cite{mi1,mi2}. As a basis of the system
Mimachi~\cite{mi2} introduced a family of Schur polynomials different from Matsuo's polynomials,
and he calculated the entries of the coefficient matrix explicitly when $n=1,2$ and $3$.
\end{Remark}

From Proposition \ref{prop:Matsuo}, if we want to know the coefficient matrices $K_i$ of the above $q$-difference systems,
it suffices to give the explicit expression for the transition matrix $R$ or its inverse $R^{-1}$.

\begin{Theorem}[\cite{BM,KOY}]\label{thm:R=LDU}
The matrix $R$ is written in terms of its Gauss matrix decomposition as
\begin{equation}\label{eq:R=LDU}
R=L_\tinyR\,D_\tinyR\,U_\tinyR
=U'{}_{\!\!\!\tinyR}\,D'{}_{\!\!\tinyR}\,L'{}_{\!\!\tinyR},
\end{equation}
where $L_\tinyR=\big(l^\tinyR_{ij}\big)_{0\le i,j\le n}$,
$D_\tinyR=\big(d^\tinyR_j \delta_{ij}\big)_{0\le i,j\le n}$,
$U_\tinyR=\big(u^\tinyR_{ij}\big)_{0\le i,j\le n}$ are
the lower triangular, diagonal, upper triangular matrices, respectively, given by
\begin{subequations}
\begin{gather}
l^\tinyR_{ij}=
\qbin{n-j}{n-i}{t^{-1}}
\frac{(-1)^{i-j}t^{-{i-j\choose 2}}\big(a_2b_2t^{j};t\big)_{i-j}}{\big(a_1^{-1}a_2 t^{-(n-2j-1)};t\big)_{i-j}},
\label{eq:lRij}\\
d^\tinyR_j=\frac{\big(a_1a_2^{-1}t^{-j};t\big)_{n-j}(a_2b_1;t)_{j}}{(a_1b_2;t)_{n-j}\big(a_1^{-1}a_2t^{-(n-j)};t\big)_{j}},
\label{eq:dRij}\\
u^\tinyR_{ij}=\qbin{j}{i}{t^{-1}}
\frac{\big(a_1b_1t^{n-j};t\big)_{j-i}}{\big(a_1a_2^{-1}t^{n-i-j};t\big)_{j-i}}, \label{eq:uRij}
\end{gather}
\end{subequations}
and $U'{}_{\!\!\!\tinyR}=\big(u^{\tinyR\,\prime}_{ij}\big)_{0\le i,j\le n}$,
$D'{}_{\!\!\tinyR}=\big(d^{\tinyR\,\prime}_{j} \delta_{ij}\big)_{0\le i,j\le n}$,
$L'{}_{\!\!\tinyR}=\big(l^{\tinyR\,\prime}_{ij}\big)_{0\le i,j\le n}$
are the upper triangular, diagonal, lower triangular matrices,
respectively, given by
\begin{subequations}
\begin{gather}
u^{\tinyR\,\prime}_{ij}=
\qbin{j}{i}{t}
\frac{(-1)^{j-i}t^{{j-i\choose 2}}\big(a_1^{-1}b_1^{-1}t^{-(n-i-1)};t\big)_{j-i}}{\big(b_1^{-1}b_2t^{i+j-n};t\big)_{j-i}},\label{eq:u'Rij}\\
d^{\tinyR\,\prime}_{j}=\frac{\big(b_1b_2^{-1}t^{n-2j+1};t\big)_{j}\big(a_2^{-1}b_1^{-1}t^{-(n-j-1)};t\big)_{n-j}}
{\big(a_1^{-1}b_2^{-1}t^{-(j-1)};t\big)_{j}\big(b_1^{-1}b_2t^{-(n-2j-1)};t\big)_{n-j}},\label{eq:d'Rij}\\
l^{\tinyR\,\prime}_{ij}=
\qbin{n-j}{n-i}{t}
\frac{\big(a_2^{-1}b_2^{-1}t^{-(i-1)};t\big)_{i-j}}{\big(b_1b_2^{-1} t^{n-2i+1};t\big)_{i-j}}.\label{eq:l'Rij}
\end{gather}
\end{subequations}
\end{Theorem}

One of the main aims of this paper is to give a proof of the above result, which we will do in Section \ref{section06}.
The explicit expression for $R^{-1}$ in terms of its Gauss matrix decomposition
is also presented as Corollary \ref{cor:inverseR} in Section \ref{section06}.

\begin{Remark}After completing of earlier version of this paper,
the author was informed that
Theorem~\ref{thm:R=LDU} previously appeared implicitly in \cite[Section~5]{BM} and~\cite{KOY}.
Our proof of the Theorem is, however, very different to that of~\cite{BM} and~\cite{KOY}.
\end{Remark}

From Theorem \ref{thm:R=LDU} we immediately obtain a closed-form expression for the determinant of $R$ (or $K_i$).
\begin{Corollary}
The determinant of the transition matrix $R$ evaluates as
\begin{equation*}
\det R
=d^\tinyR_0\,d^\tinyR_1\cdots d^\tinyR_n=\big({-}a_1a_2^{-1}\big)^{{n+1\choose 2}}
\prod_{i=1}^n\frac{(a_2b_1;t)_i}{(a_1b_2;t)_i}.
\end{equation*}
The determinants of the coefficient matrices $K_1$ and $K_2$ given in \eqref{eq:Ki} evaluate as
\begin{equation*}
\det K_1=\det \big(R^{-1}D_1\big)=
\big({-}a_2a_1^{-1}q^\alpha t^{n-1}\big)^{{n+1\choose 2}}\prod_{i=1}^n\frac{(a_1b_2;t)_i}{(a_2b_1;t)_i}
\end{equation*}
and
\begin{equation*}
\det K_2=\det \big(D_2\big(T_{q,b_2}^{-1}T_{q,a_2}R\big)\big)=
\big({-}a_1a_2^{-1}q^{\alpha-1} t^{n-1}\big)^{{n+1\choose 2}}
\prod_{i=1}^n\frac{(qa_2b_1;t)_i}{\big(q^{-1}a_1b_2;t\big)_i}.
\end{equation*}
\end{Corollary}

Next, we focus on the $q$-difference system with respect to the shift $\alpha\to \alpha+1$
for the Jackson integral of symmetric Selberg type. Using Matsuo's polynomials $\{e_i(a_1,b_2;z)\,|\, i=0,1,\ldots,n\}$,
this $q$-difference system is given explicitly in terms of its Gauss matrix decomposition.
\begin{Theorem}\label{thm:main} Let $T_\alpha$ be the shift operator with respect to $\alpha\to \alpha+1$, i.e., $T_\alpha f(\alpha)=f(\alpha+1)$
for an arbitrary function $f(\alpha)$ of $\alpha\in \mathbb{C}$. Then
\begin{gather}
T_{\alpha}\big(\la e_0(a_1,b_2),\xi\ra,\la e_1(a_1,b_2),\xi\ra,\ldots,\la e_n(a_1,b_2),\xi\ra\big)\nonumber\\
\qquad{}=\big(\la e_0(a_1,b_2),\xi\ra,\la e_1(a_1,b_2),\xi\ra,\ldots,\la e_n(a_1,b_2),\xi\ra\big)A,\label{eq:main}
\end{gather}
where the coefficient matrix $A$ is written in terms of its Gauss matrix decomposition as
\begin{equation*}
A=L_\tinyA\,D_\tinyA\,U_\tinyA
=U'{}_{\!\!\!\tinyA}\,D'{}_{\!\!\tinyA}\,L'{}_{\!\!\tinyA}.
\end{equation*}
Here
$L_\tinyA=\big(l^\tinyA_{ij}\big)_{0\le i,j\le n}$,
$D_\tinyA=\big(d^\tinyA_j\,\delta_{ij}\big)_{0\le i,j\le n}$,
$U_\tinyA=\big(u^\tinyA_{ij}\big)_{0\le i,j\le n}$ are
the lower triangular, diagonal, upper triangular matrices, respectively, given by
\begin{subequations}
\begin{gather}
l^\tinyA_{ij}
=(-1)^{i-j}
t^{{n-i\choose 2}-{n-j\choose 2}}
\qbin{n-j}{n-i}{t}
\frac{\big(a_2b_2 t^j;t\big)_{i-j}}
{\big(q^\alpha a_2b_2 t^{2j};t\big)_{i-j}},\label{eq:lAij}
\\
d^\tinyA_j=
a_1^{n-j}a_2^jt^{{j\choose 2}+{n-j\choose 2}}
\frac{\big(q^\alpha;t\big)_{j}\big(q^\alpha a_2b_2 t^{2j};t\big)_{n-j}}
{\big(q^\alpha a_2b_2 t^{j-1};t\big)_{j}\big(q^\alpha a_1a_2b_1b_2 t^{n+j-1};t\big)_{n-j}},
\label{eq:dAij}\\
u^\tinyA_{ij}
=\big({-}q^\alpha a_1^{-1}a_2\big)^{j-i}t^{{j\choose 2}-{i\choose 2}}
\qbin{j}{i}{t}
\frac{\big(a_1b_1t^{n-j};t\big)_{j-i}}
{\big(q^\alpha a_2b_2 t^{2i};t\big)_{j-i}},\label{eq:uAij}
\end{gather}
\end{subequations}
and $U'{}_{\!\!\!\tinyA}=\big(u^{\tinyA\,\prime}_{ij}\big)_{0\le i,j\le n}$,
$D'{}_{\!\!\tinyA}=\big(d^{\tinyA\,\prime}_{j}\,\delta_{ij}\big)_{0\le i,j\le n}$, $L'{}_{\!\!\tinyA}=\big(l^{\tinyA\,\prime}_{ij}\big)_{0\le i,j\le n}$ are the upper triangular, diagonal, lower triangular matrices,
respectively, given by
\begin{subequations}
\begin{gather}
u^{\tinyA\,\prime}_{ij}
=\big({-}q^\alpha\big)^{j-i}t^{{n-i\choose 2}-{n-j\choose 2}}
\qbin{j}{i}{t}
\frac{\big(a_1b_1t^{n-j};t\big)_{j-i}
}{\big(q^\alpha a_1b_1 t^{2(n-j)};t\big)_{j-i}},\label{eq:u'Aij}\\
d^{\tinyA\,\prime}_{j}
=a_1^{n-j}a_2^jt^{{j\choose 2}+{n-j\choose 2}}
\frac{\big(q^\alpha a_1b_1 t^{2(n-j)};t\big)_j\big(q^\alpha;t\big)_{n-j}
}{\big(q^\alpha a_1a_2b_1b_2 t^{2n-j-1};t\big)_j\big(q^\alpha a_1b_1 t^{n-j-1};t\big)_{n-j}},\label{eq:d'Aij}\\
l^{\tinyA\,\prime}_{ij}=\big({-}a_1a_2^{-1}\big)^{i-j}t^{{j\choose 2}-{i\choose 2}}
\qbin{n-j}{n-i}{t}\frac{\big(a_2b_2t^j;t\big)_{i-j}}
{\big(q^\alpha a_1b_1 t^{2(n-i)};t\big)_{i-j}}.\label{eq:l'Aij}
\end{gather}
\end{subequations}
\end{Theorem}
The first part of Theorem \ref{thm:main} will be proved in Section~\ref{section05}, while the latter part of
Theorem~\ref{thm:main} will be explained in the Appendix.
Note that, from this theorem we immediately have the following.

\begin{Corollary}The determinant of the coefficient matrix $A$ evaluates as
\begin{equation*}
\det A=d^\tinyA_0\,d^\tinyA_1\cdots d^\tinyA_n=
(a_1a_2)^{{n+1\choose 2}}t^{2{n+1\choose 3}}\prod_{i=1}^n
\frac{\big(q^\alpha;t\big)_{i}}{\big(q^\alpha a_1a_2b_1b_2 t^{2n-i-1};t\big)_i}.
\end{equation*}
\end{Corollary}

\begin{Remark}The expression \eqref{eq:main} for the $q$-difference equation is equivalent to
\begin{gather*}
T_{\alpha}\big(\la e_0(a_1,b_2),\xi\ra,\la e_1(a_1,b_2),\xi\ra,\ldots,\la e_n(a_1,b_2),\xi\ra\big)
U_\tinyA^{-1}\nonumber\\
\qquad{}=\big(\la e_0(a_1,b_2),\xi\ra,\la e_1(a_1,b_2),\xi\ra,\ldots,\la e_n(a_1,b_2),\xi\ra\big)L_\tinyA D_\tinyA
\end{gather*}
or
\begin{gather*}
T_{\alpha}\big(\la e_0(a_1,b_2),\xi\ra,\la e_1(a_1,b_2),\xi\ra,\ldots,\la e_n(a_1,b_2),\xi\ra\big)
L'{}_{\!\!\tinyA}^{-1}
\nonumber\\
\qquad=\big(\la e_0(a_1,b_2),\xi\ra,\la e_1(a_1,b_2),\xi\ra,\ldots,\la e_n(a_1,b_2),\xi\ra\big)U'{}_{\!\!\!\tinyA}D'{}_{\!\!\tinyA}.
\end{gather*}
Here the entries of $U_\tinyA^{-1}$ and $L'{}_{\!\!\tinyA}^{-1}$ are also factorized into binomials
like $U_\tinyA$ and $L'{}_{\!\!\tinyA}$, respectively.
We will see the explicit expression for $U_\tinyA^{-1}$ in Section~\ref{section04} as Proposition~\ref{prop:inverseU_A}.
For the explicit expression for $L'{}_{\!\!\tinyA}$, see Proposition~\ref{prop:inverseL'_A}.
\end{Remark}

\begin{Remark}
If we consider the $q\to 1$ limit after replacing $a_i$ and $b_i$
as $a_1=1$, $a_2=x$, $b_1= q^\beta$ and $b_2= q^{\gamma}x^{-1}$ on Theorem \ref{thm:main},
then $e_i\big(1,q^{\gamma}x^{-1};z\big)\Phi(z)\Delta(z)\frac{{\rm d}_qz_1}{z_1}\cdots \frac{{\rm d}_qz_n}{z_n}$ tends to
$e'_i(z)\Psi'(z){\rm d}z_1\cdots {\rm d}z_n$, where
\begin{gather*}
\Psi'(z) =\prod_{i=1}^n z_i^{\alpha-1}(1-z_i)^{\beta-1}
(1-z_i/x)^{\gamma-1}
\prod_{1\le j<k\le n}(z_j-z_k)^{2\tau},\\
e'_i(z) ={\cal A}\left(\prod_{j=1}^{n-i}(1-z_j/x)\prod_{j=n-i+1}^n(1-z_j)\right).
\end{gather*}
This confirms that Theorem~\ref{thm:main} in the $q\to 1$ limit is consistent with the result presented in
Proposition~\ref{prop:FI10}.
\end{Remark}

The paper is organized as follows. After defining some basic terminology in Section~\ref{section02},
we characterize in Section~\ref{section03}
Matsuo's polynomials by their vanishing property (Proposition~\ref{prop:matsuo2}), and
define a family of symmetric polynomials of higher degree, which includes Matsuo's polynomials.
We call such polynomials the {\it interpolation polynomials}, which are
inspired from Aomoto's method~\cite[Section~8]{AAR}, \cite{Ao87},
which is a technique to obtain difference equations for the Selberg integrals
(see also~\cite{IF2017} for a $q$-analogue of Aomoto's method).
We state several vanishing properties
for the interpolation polynomials, which are used in subsequent sections.
In Section~\ref{section04} we present three-term relations (Lemma~\ref{lem:3term1st})
among the interpolation polynomials.
These are key equations for obtaining the coefficient matrix of the $q$-difference system with respect to the shift $\alpha\to \alpha+1$. By repeated use of these three-term relations we obtain a proof of Theorem~\ref{thm:main}.
Section~\ref{section05} is devoted to the proof of Lemma~\ref{lem:3term1st}.
In Section \ref{section06} we explain the Gauss decomposition of the transition matrix~$R$.
For this purpose, we introduce another set of symmetric polynomials called the {\it Lagrange interpolation polynomials of type A}
in~\cite{IN2018},
which are different from Matsuo's polynomials. Both upper and lower triangular matrices in the decomposition can be understood as a transition matrix between Matsuo's polynomials and the other polynomials.
In the Appendix we explain the proof of the latter part of Theorem~\ref{thm:main}.

Finally we would like to make some remarks about the original motivation for the current paper.
Although the author already knew the results of this paper before publishing~\cite{FI10},
many years have passed since then.
The author recently learned of an interesting
application of the $q$-difference systems of this paper
in collaboration with Yasuhiko Yamada.
They intend to publish the detail in a forthcoming paper.

\section{Notation}\label{section02}

Let $S_n$ be the symmetric group on $\{1,2,\ldots, n\}$.
For a function $f\colon (\mathbb{C}^*)^n\to \mathbb{C}$
we define an action of the symmetric group $S_n$ on $f$ by
\begin{equation*}
(\sigma f)(z):=f\big(\sigma^{-1}(z)\big)=f(z_{\sigma(1)},z_{\sigma(2)},\ldots,z_{\sigma(n)})
\qquad\text{for}\quad \sigma\in S_n.
\end{equation*}
We say that a function $f(z)$ on $({\mathbb{C}^*})^n$
is {\it symmetric} or {\it skew-symmetric}
if $\sigma f(z)=f(z)$ or $\sigma f(z)=(\operatorname{sgn}\sigma ) f(z)$
for all $\sigma \in S_n$, respectively.
We denote by ${\cal A} f(z)$
the alternating sum over~$S_n$ defined by
\begin{equation}\label{eq:00Af}
({\cal A} f)(z):=\sum_{\sigma\in S_n}(\operatorname{sgn} \sigma) (\sigma f)(z),
\end{equation}
which is skew-symmetric.
Let $P$ be the set of partitions defined by
\begin{equation*}
P:=\big\{(\lambda_1,\lambda_2,\ldots,\lambda_n)\in {\mathbb Z}^n \,|\,\lambda_1\ge \lambda_2\ge \cdots \ge \lambda_n \ge 0 \big\}.
%\label{eq:calP}
\end{equation*}
We define the {\em lexicographic ordering} $<$ on $P$ as follows.
For $\lambda, \mu\in P$, we denote $\lambda < \mu$ if
there exists a positive integer $k$ such that
$\lambda_i=\mu_i$ for all $i<k$ and $\lambda_k<\mu_k$.
For $\lambda=(\lambda_1,\lambda_2,\ldots,\lambda_n) \in {\mathbb Z}^n$, we denote by $z^\lambda$ the monomial
$z_1^{\lambda_1}z_2^{\lambda_2}\cdots z_n^{\lambda_n}$.
For $\lambda\in P$ the {\it monomial symmetric polynomials} $m_\lambda (z)$ are defined by
\begin{equation*}
m_\lambda (z):=\sum_{\mu\in S_n\lambda}z^{\mu},
\end{equation*}
where $S_n\lambda:=\{\sigma\lambda\,|\, \sigma\in S_n\}$ is the $S_n$-orbit of $\lambda$.
For $\lambda\in P$, we denote by $m_i$ the multiplicity of $i$ in $\lambda$,
i.e., $m_i=\#\{j\,|\,\lambda_j=i\}$, see \cite{Mac} for instance.
It is convenient to use the notation
$\lambda=\big(1^{m_1}2^{m_2}\cdots r^{m_r}\cdots\big)$: for example,
$\big(1^32^2\big)=(2,2,1,1,1,0)$
and $z^{(1^32^2)}=z_1^2z_2^2z_3z_4z_5$.

\section{Interpolation polynomials}\label{section03}

In this section we define a family of symmetric functions which extends Matsuo's polynomials.
For $a, b\in \mathbb{C}^*$ and $z=(z_1,\ldots,z_n)\in (\mathbb{C}^*)^n$
let $E_{k,i}(a,b;z)$ $(k,i=0,1,\ldots,n)$ be functions specified by
\begin{equation}\label{eq:Eki1}
E_{k,i}(a,b;z):=z_1z_2\cdots z_k\Delta(t;z)
\prod_{j=1}^{n-i}(1-bz_j)\prod_{j=n-i+1}^n\big(1-a^{-1}z_j\big),
\end{equation}
where
\begin{equation*}
\Delta(t;z):=\prod_{1\le i<j\le n}\big(z_i-t^{-1}z_j\big)=t^{-{n\choose 2}}\prod_{1\le i<j\le n}(tz_i-z_j),
\end{equation*}
and let $\tilde{E}_{k,i}(a,b;z)$ $(k,i=0,1,\ldots,n)$ be the symmetric functions of $z\in (\mathbb{C}^*)^n$ specified by
\begin{equation}
\label{eq:Eki2}
\tilde{E}_{k,i}(a,b;z):=\frac{{\cal A}E_{k,i}(a,b;z)}{\Delta(z)},
\end{equation}
which, in particular, satisfy
\begin{equation*}
\tilde{E}_{0,i}(a,b;z)=e_i(a,b;z)\qquad\text{and}\qquad
\tilde{E}_{n,i}(a,b;z)=z_1z_2\cdots z_n e_i(a,b;z),
\end{equation*}
as special cases.
We sometimes abbreviate $\tilde{E}_{k,i}(a,b;z)$ to $\tilde{E}_{k,i}(z)$.
The leading term of the symmetric polynomial $\tilde{E}_{k,i}(z)$ is $m_{(1^{n-k}2^k)}(z)$, i.e.,
\begin{equation*}
\tilde{E}_{k,i}(z)=C_{ki} m_{(1^{n-k}2^k)}(z)+\text{lower order terms},
\end{equation*}
where the coefficient $C_{ki}$ of the monomial $m_{(1^{n-k}2^k)}(z)$ is expressed as
\begin{equation*}
C_{ki}=(-1)^n\frac{\big(t^{-1};t^{-1}\big)_k\big(t^{-1};t^{-1}\big)_{n-k}}{a^{i}b^{-(n-i)}\big(1-t^{-1}\big)^n}.
\end{equation*}
For arbitrary $x,y\in {\mathbb C}^*$, we set
\begin{equation}
\zeta_j(x,y):=\big(\underbrace{yt^{-(j-1)},yt^{-(j-2)},\ldots,yt^{-1},y\phantom{\Big|}\!\!}_j, \underbrace{x,xt,xt^2,\ldots,xt^{n-j-1}\phantom{\Big|}\!\!}_{n-j}\big)\in ({\mathbb C}^*)^n.\label{eq:zeta(xy)}
\end{equation}

The following gives another characterization of Matsuo's polynomials $e_i(a,b;z)=\tilde{E}_{0,i}(z)$.
\begin{Proposition}\label{prop:matsuo2}
The leading term of the function $\tilde{E}_{0,i}(z)$ is $m_{(1^n)}(z)$
up to a multiplicative constant.
The functions $\tilde{E}_{0,i}(z)$, $i=0,1,\ldots,n$, satisfy
\begin{equation}\label{eq:matsuo2}
\tilde{E}_{0,i}\big(\zeta_j\big(a,b^{-1}\big)\big)=c_i\delta_{ij},
\end{equation}
where the constant $c_i$ is given by
\begin{subequations}
\begin{align}
c_i&=\big(abt^{i};t\big)_{n-i}\big(a^{-1}b^{-1}t^{-(i-1)};t\big)_{i}
\frac{(t;t)_i(t;t)_{n-i}}
{t^{n\choose 2}(1-t)^n}
\label{eq:c-i-a}\\
&=
(ab;t)_{n-i}\big(a^{-1}b^{-1}t^{-(n-1)};t\big)_{i}
\frac{\big(t^{-1};t^{-1}\big)_i\big(t^{-1};t^{-1}\big)_{n-i}}
{\big(1-t^{-1}\big)^n}.\label{eq:c-i-b}
\end{align}
\end{subequations}
\end{Proposition}

\begin{Remark}
The set of symmetric functions $\big\{\tilde{E}_{0,i}(z)\,|\,i=0,1,\ldots,n\big\}$
forms a basis of the linear space spanned by $\{m_\lambda(z)\,|\,\lambda\le (1^n)\}$.
Conversely such basis satisfying the condition~(\ref{eq:matsuo2}) is uniquely determined.
Thus we can take Proposition~\ref{prop:matsuo2} as a definition of Matsuo's polynomials,
instead of~(\ref{eq:matsuo}).
\end{Remark}

\begin{proof}
By definition
$\tilde{E}_{0,i}(\zeta_j(a,b^{-1}))=0$ if $i\ne j$.
$\tilde{E}_{0,i}(\zeta_i(a,b^{-1}))$
evaluates as
\begin{equation*}
\tilde{E}_{0,i}(\zeta_i(a,b^{-1}))=
\frac{E_{0,i}\big(at^{n-i-1},\ldots,at^2,at,a,b^{-1},b^{-1}t^{-1},\ldots,b^{-1}t^{-(i-1)}\big)}
{\Delta\big(at^{n-i-1},\ldots,at^2,at,a,b^{-1},b^{-1}t^{-1},\ldots,b^{-1}t^{-(i-1)}\big)},
\end{equation*}
which coincides with (\ref{eq:c-i-a}) and (\ref{eq:c-i-b}).
\end{proof}

\begin{Lemma}[triangularity]\label{lem:triangularity}
Suppose
\begin{equation*}
\xi_j:=\big(\underbrace{b^{-1}t^{-(j-1)},b^{-1}t^{-(j-2)},\ldots,b^{-1}t^{-1},b^{-1}\phantom{\Big|}\!\!}_j,z_1,z_2,\ldots,z_{n-j}\big)\in ({\mathbb C}^*)^n.
\end{equation*}
Then
\begin{equation}\label{eq:tri-xi1}
\tilde{E}_{k,i}(\xi_j)=0\qquad\text{if}\quad 0\le i<j\le n.
\end{equation}
Moreover, $\tilde{E}_{0,i}(\xi_i)$ evaluates as
\begin{align}
\tilde{E}_{0,i}(\xi_i)
&=
\frac{t^{-i(n-i)}\big(t^{-1};t^{-1}\big)_i\big(t^{-1};t^{-1}\big)_{n-i}}
{\big(1-t^{-1}\big)^n}\big(a^{-1}b^{-1}t^{-(i-1)};t\big)_{i}
\prod_{j=1}^{n-i}\big(1-z_jbt^i\big)\nonumber\\
&=\frac{(t;t)_i(t;t)_{n-i}}
{t^{n\choose 2}(1-t)^n}
\big(a^{-1}b^{-1}t^{-(i-1)};t\big)_{i}
\prod_{j=1}^{n-i}\big(1-z_jbt^i\big).\label{eq:tri-xi2}
\end{align}
On the other hand, if
$\eta_j:=\big(z_1,z_2,\ldots,z_j,\underbrace{a,at,at^2,\ldots,at^{n-j-1}\phantom{\Big|}\!\!}_{n-j}\big)\in ({\mathbb C}^*)^n,$ then
\begin{equation}\label{eq:tri-eta1}
\tilde{E}_{k,i}(\eta_j)=0\qquad\text{if}\quad 0\le j<i\le n.
\end{equation}
Moreover, $\tilde{E}_{0,i}(\eta_i)$ evaluates as
\begin{equation}
\label{eq:tri-eta2}
\tilde{E}_{0,i}(\eta_i)=\frac{\big(t^{-1};t^{-1}\big)_i\big(t^{-1};t^{-1}\big)_{n-i}}{\big(1-t^{-1}\big)^n}(ab;t)_{n-i}
\prod_{j=1}^i \big(1-z_j/at^{n-i}\big).
\end{equation}
\end{Lemma}
\begin{proof}By the definition (\ref{eq:Eki1}) of $\tilde{E}_{k,i}(z)$
it immediately follows that
$\tilde{E}_{k,i}(\xi_j)=0$ if $i<j$, and $\tilde{E}_{k,i}(\eta_j)=0$ if $j<i$.
If we put $z_j=b^{-1}t^{-i}$ $(j=1,2,\ldots,n-i)$
in the polynomial $\tilde{E}_{k,i}(\xi_i)$,
then we have $\tilde{E}_{k,i}(\xi_i)=0$ because $\tilde{E}_{k,i}(\xi_i)$ satisfies the condition of~(\ref{eq:tri-xi1}).
This implies that $\tilde{E}_{k,i}(\xi_i)$ is divisible by $\prod_{j=1}^{n-i}(1-z_jbt^i)$
up to a constant.
Thus we write $\tilde{E}_{k,i}(\xi_i)=c\prod_{j=1}^{n-i}(1-z_jbt^i)$,
where $c$ is some constant independent of $z_1,\ldots,z_{n-i}$. Next we determine the explicit form of $c$.
If we put $z_1=a, z_2=a t, \ldots, z_{n-i}=a t^{n-i-1}$ in $\tilde{E}_{k,i}(\xi_i)$, then
$\tilde{E}_{k,i}(\xi_i)=\tilde{E}_{0,i}\big(\zeta_i\big(a,b^{-1}\big)\big)$.
From~(\ref{eq:matsuo2}), we have $c_i=c(abt^i;t)_{n-i}$, where $c_i$ is given by~(\ref{eq:c-i-b}).
Therefore the constant~$c$ is evaluated as $c=c_i/(abt^i;t)_{n-i}$, i.e.,
we obtain the expression~(\ref{eq:tri-xi2}) for $\tilde{E}_{k,i}(\xi_i)$.
The evaluation (\ref{eq:tri-eta2}) is carried out in the same way as above.
\end{proof}

\begin{Lemma}\label{lem:tri-zeta(xb)}
For $0\le j\le n$, let $\zeta_j\big(x,b^{-1}\big)\in ({\mathbb C}^*)^n$ be the point specified by {\rm (\ref{eq:zeta(xy)})} with $y=b^{-1}$.
Then
\begin{equation}\label{eq:tri-zeta(xb)2}
\tilde{E}_{0,i}(\zeta_{j}(x,b^{-1}))=
\frac{\big(xbt^{i};t\big)_{n-i}\big(xa^{-1};t\big)_{i-j}\big(a^{-1}b^{-1}t^{-(j-1)};t\big)_{j}(t;t)_n}
{t^{n\choose 2}(1-t)^n}
\frac{\qbin{i}{j}{t}}{\qbin{n}{j}{t}}.
\end{equation}
Moreover, if $i+k\le n$ $($i.e., $k\le n-i\le n-j)$, then
\begin{equation}\label{eq:tri-zeta(xb)3}
\tilde{E}_{k,i}\big(\zeta_{j}\big(x,b^{-1}\big)\big)
=x^kt^{(n-j)k-{k+1\choose 2}}\tilde{E}_{0,i}\big(\zeta_{j}\big(x,b^{-1}\big)\big).
\end{equation}
\end{Lemma}

\begin{proof}
If $i\le j$, $\tilde{E}_{k,i}\big(\zeta_{j}\big(x,b^{-1}\big)\big)=0$ is a special case of~(\ref{eq:tri-xi1}).
Suppose $j\le i$. If we put $x=b^{-1}t^{-i-k} $ $(k=0,1,\ldots,n-i-1)$, then
the polynomial $\tilde{E}_{0,i}\big(\zeta_j\big(x,b^{-1}\big)\big)$
satisfies
the condition (\ref{eq:tri-xi1}) of Lemma~\ref{lem:triangularity}, so that it is equal to zero,
which implies $\tilde{E}_{0,i}\big(\zeta_j\big(x,b^{-1}\big)\big)$ is divisible by $(xbt^{i};t)_{n-i}$.
If we put $x=at^{-k} $ $(k=0,1,\ldots,i-j-1)$, then the polynomial $\tilde{E}_{0,i}\big(\zeta_j\big(x,b^{-1}\big)\big)$
satisfies the condition~(\ref{eq:tri-eta1}) of Lemma~\ref{lem:triangularity}, so that it is also equal to zero,
which implies $\tilde{E}_{0,i}\big(\zeta_j\big(x,b^{-1}\big)\big)$ is divisible by $\big(xa^{-1};t\big)_{i-j}$. Therefore we have
\begin{equation}\label{eq:E0i-zeta(xb)}
\tilde{E}_{0,i}\big(\zeta_j\big(x,b^{-1}\big)\big)=c\big(xbt^{i};t\big)_{n-i}\big(xa^{-1};t\big)_{i-j},
\end{equation}
where $c$ is some constant independent of $x$. Next we determine the explicit form of~$c$.
Put $x=b^{-1}t^{-(i-1)}$ in (\ref{eq:E0i-zeta(xb)}).
Then, using~(\ref{eq:tri-xi2}), the left-hand side of~(\ref{eq:E0i-zeta(xb)}) is written as
\begin{align}
\tilde{E}_{0,i}\big(\zeta_j\big(x,b^{-1}\big)\big)\Big|_{x=b^{-1}t^{-(i-1)}}
&=\tilde{E}_{0,i}\big(\zeta_i\big(x,b^{-1}\big)\big)\Big|_{x=b^{-1}t^{-(j-1)}}\nonumber\\
&=\frac{(t;t)_i(t;t)_{n-i}(t;t)_{n-j}}
{t^{n\choose 2}(1-t)^n(t;t)_{i-j}}
\big(a^{-1}b^{-1}t^{-(i-1)};t\big)_{i},\label{eq:E0i-zeta(xb)l}
\end{align}
while the right-hand side of (\ref{eq:E0i-zeta(xb)}) is
\begin{equation}\label{eq:E0i-zeta(xb)r}
c\big(xbt^{i};t\big)_{n-i}\big(xa^{-1};t\big)_{i-j}\Big|_{x=b^{-1}t^{-(i-1)}}
=c(t;t)_{n-i}\big(a^{-1}b^{-1}t^{-(i-1)};t\big)_{i-j}.
\end{equation}
Comparing with (\ref{eq:E0i-zeta(xb)l}) and (\ref{eq:E0i-zeta(xb)r}), we have
\begin{equation*}
c=\frac{(t;t)_i(t;t)_{n-j}}
{t^{n\choose 2}(1-t)^n(t;t)_{i-j}}
\big(a^{-1}b^{-1}t^{-(j-1)};t\big)_{j}.
\end{equation*}
Therefore we obtain (\ref{eq:tri-zeta(xb)2}).
Moreover, if $i+k\le n$, by definition we have
\begin{equation*}
\tilde{E}_{k,i}\big(\zeta_{j}\big(x,b^{-1}\big)\big)
=
\big(xt^{(n-j-1)}\big)\big(xt^{(n-j-2)}\big)\cdots\big(xt^{(n-j-k)}\big)
\tilde{E}_{0,i}\big(\zeta_{j}\big(x,b^{-1}\big)\big),
\end{equation*}
which coincides with (\ref{eq:tri-zeta(xb)3}).
\end{proof}

As a counterpart of Lemma \ref{lem:tri-zeta(xb)}, we have the following.
\begin{Lemma}\label{lem:tri-zeta(ay)}
For $0\le j\le n$, let $\zeta_j(a,y)\in ({\mathbb C}^*)^n$ be the point specified by \eqref{eq:zeta(xy)} with $x=a$.
Then
\begin{equation}\label{eq:tri-zeta(ay)2}
\tilde{E}_{0,i}(\zeta_j(a,y))=
\frac{\big(ybt^{-(j-i-1)};t\big)_{j-i}\big(y/at^{n-1};t\big)_i(ab;t)_{n-j}\big(t^{-1};t^{-1}\big)_{n}}{\big(1-t^{-1}\big)^n}
\frac{\qbin{n-i}{n-j}{t^{-1}}}
{\qbin{n}{j}{t^{-1}}}.
\end{equation}
Moreover, if $n\le i+k$ $($i.e., $n-j\le n-i\le k)$, then
\begin{equation*}%\label{eq:tri-zeta(ay)3}
\tilde{E}_{k,i}(\zeta_j(a,y))
=y^{k+j-n}a^{n-j}t^{{n-j\choose 2 }-{k+j-n\choose 2 }}
\tilde{E}_{0,i}(\zeta_j(a,y)).
\end{equation*}
\end{Lemma}

\begin{proof}
This lemma can be proved in the same way as Lemma~\ref{lem:tri-zeta(xb)} using Lemma~\ref{lem:triangularity}.
\end{proof}

\section{Three-term relations}\label{section04}

In this section we fix $\tilde{E}_{k,i}(z)=\tilde{E}_{k,i}(a_1, b_2;z)$. In particular
we have $e_i(a_1, b_2;z)=\tilde{E}_{0,i}(a_1, b_2;z)$.
The following lemma is a technical key for computing the coefficient matrix $A$ of (\ref{eq:main}) in Theorem \ref{thm:main}.
We abbreviate $\big\la \tilde{E}_{k,i},x\big\ra$ to $\big\la \tilde{E}_{k,i}\big\ra$ throughout this section.
\begin{Lemma}[three-term relations]
\label{lem:3term1st}
Suppose $i+k\le n$. Then,
\begin{gather}\label{eq:3term01}
\frac{1-q^\alpha a_1 a_2 b_1 b_2 t^{2n-k-1}}{a_1t^{n-i-k}}
\big\la \tilde{E}_{k,i}\big\ra=\big(1-q^\alpha a_2 b_2 t^{n+i-k}\big)\big\la \tilde{E}_{k-1,i}\big\ra-\big(1-a_2 b_2 t^{i}\big)\big\la \tilde{E}_{k-1,i+1}\big\ra.\!\!\!\!
\end{gather}
On the other hand, if $i+k\ge n$, then,
\begin{gather}
t^{n-k}\big(1-q^\alpha t^{n-k}\big)\big\la \tilde{E}_{k-1,i+1}\big\ra\nonumber\\
\qquad{}=
  a_1^{-1} q^\alpha t^{n+i-k}\big(1- a_1 b_1 t^{n-i-1}\big) \big\la \tilde{E}_{k,i}\big\ra+
 a_2^{-1}\big(1-q^\alpha a_2 b_2 t^{n+i-k}\big) \big\la \tilde{E}_{k,i+1}\big\ra.\label{eq:3term02}
\end{gather}
\end{Lemma}

The proof of this lemma will be given in Section~\ref{section05}.
The rest of this section is devoted to computing
the Gauss matrix decomposition of $A$ in Theorem~\ref{thm:main} using Lemma~\ref{lem:3term1st}.
By repeated use of Lemma \ref{lem:3term1st}, we have the following.
\begin{Corollary}\label{cor:<Eki>}
Suppose $i+k\le n$. For $0\le l\le k$, $\big\la\tilde{E}_{k,i}\big\ra$ is expressed as
\begin{gather}\label{eq:<Eki>01}
\big\la \tilde{E}_{k,i}\big\ra=\sum_{j=0}^l L_{k-l,i+j}^{k,i}\big\la \tilde{E}_{k-l,i+j}\big\ra,
\end{gather}
where the coefficients $L_{k-l,i+j}^{k,i}$ is expressed as
\begin{equation*}
L_{k-l,i+j}^{k,i}
=
\qbin{l}{j}{t}
\frac{(-1)^j\big(a_1t^{n-i-k}\big)^lt^{{l-j\choose 2}}
\big(a_2b_2 t^i;t\big)_{j}\big(q^\alpha a_2b_2 t^{n+i+j-k};t\big)_{l-j}}
{\big(q^\alpha a_1a_2b_1b_2 t^{2n-k-1};t\big)_{l}}.
\end{equation*}
On the other hand, if $i+k\ge n$, then,
\begin{equation}\label{eq:<Eki>02}
\big\la \tilde{E}_{k,i}\big\ra=\sum_{j=0}^lU^{k,i}_{k-l+j,i-j}\big\la \tilde{E}_{k-l+j,i-j}\big\ra,
\end{equation}
where the coefficient $U^{k,i}_{k-l+j,i-j}$ is expressed as
\begin{equation*}
U^{k,i}_{k-l+j,i-j}
=\qbin{l}{j}{t}
\frac{\big({-}q^\alpha a_1^{-1}t^{i-l}\big)^j\big(a_2t^{n-k}\big)^lt^{{l\choose 2}}\big(a_1b_1t^{n-i};t\big)_j\big(q^\alpha t^{n-k};t\big)_{l-j}}
{\big(q^\alpha a_2b_2 t^{n+i-k-j-1};t\big)_{l-j}\big(q^\alpha a_2b_2 t^{n+i-k-2j+l};t\big)_{j}}.
\end{equation*}
\end{Corollary}

\begin{proof}\eqref{eq:<Eki>01} and \eqref{eq:<Eki>02} follow by double induction on $l$ and $j$
using \eqref{eq:3term01} and \eqref{eq:3term02}, respecti\-ve\-ly.
\end{proof}

In particular, we immediately have the following as a special case of Corollary~\ref{cor:<Eki>}.

\begin{Corollary} For $0\le j\le n$, $\la E_{n-j,j}\ra$ is expressed as
\begin{equation}\label{eq:<En-j,j>}
\big\la \tilde{E}_{n-j,j}\big\ra=\sum_{i=j}^n\tilde{l}_{ij}\big\la \tilde{E}_{0,i}\big\ra,
\end{equation}
where the coefficients $\tilde{l}_{ij}$ is expressed as
\begin{equation}\label{eq:<En-j,j>c}
\tilde{l}_{ij}=L_{0,i}^{n-j,j}
=
\qbin{n-j}{n-i}{t}
\frac{(-1)^{i-j}a_1^{n-j}t^{{n-i\choose 2}}
\big(a_2b_2 t^j;t\big)_{i-j}\big(q^\alpha a_2b_2 t^{i+j};t\big)_{n-i}}
{\big(q^\alpha a_1a_2b_1b_2 t^{n+j-1};t\big)_{n-j}},
\end{equation}
while, for $0\le j\le n$, $\la E_{n,j}\ra$ is expressed as
\begin{equation}\label{eq:<En,j>}
\big\la \tilde{E}_{n,j}\big\ra=\sum_{i=0}^j\tilde{u}_{ij}\big\la \tilde{E}_{n-i,i}\big\ra,
\end{equation}
where the coefficients $\tilde{u}_{ij}$ is expressed as
\begin{equation}\label{eq:<En,j>c}
\tilde{u}_{ij}=U^{n,j}_{n-i,i}=\qbin{j}{i}{t}
\frac{\big({-}q^\alpha a_1^{-1}\big)^{j-i}a_2^jt^{{j\choose 2}}
\big(q^\alpha;t\big)_{i}\big(a_1b_1t^{n-j};t\big)_{j-i}}
{\big(q^\alpha a_2b_2 t^{i-1};t\big)_{i}\big(q^\alpha a_2b_2 t^{2i};t\big)_{j-i}}.
\end{equation}
\end{Corollary}

We now give the explicit expression for $A$ in terms of Gauss decomposition.

\begin{proof}[Proof of Theorem \ref{thm:main}]
From (\ref{eq:<En-j,j>}), we have
\begin{equation*}
\big(\big\la \tilde{E}_{n,0}\big\ra,\big\la \tilde{E}_{n-1,1}\big\ra,\ldots,\big\la \tilde{E}_{1,n-1}\big\ra,\big\la \tilde{E}_{0,n}\big\ra\big)
=\big(\big\la \tilde{E}_{0,0}\big\ra,\big\la \tilde{E}_{0,1}\big\ra,\ldots,\big\la \tilde{E}_{0,n-1}\big\ra,\big\la \tilde{E}_{0,n}\big\ra\big)\tilde{L},
\end{equation*}
where the matrix $\tilde{L}=\big(\tilde{l}_{ij}\big)_{0\le i,j\le n}$ is defined by \eqref{eq:<En-j,j>c}.
Moreover, from \eqref{eq:<En,j>} we have
\begin{subequations}
\begin{align}
\big(\big\la \tilde{E}_{n,0}\big\ra,\big\la \tilde{E}_{n,1}\big\ra,\ldots,\big\la \tilde{E}_{n,n-1}\big\ra,\big\la \tilde{E}_{n,n}\big\ra\big)
&=
\big(\big\la \tilde{E}_{n,0}\big\ra,\big\la \tilde{E}_{n-1,1}\big\ra,\ldots,\big\la \tilde{E}_{1,n-1}\big\ra,\big\la \tilde{E}_{0,n}\big\ra\big)\tilde{U}
\label{eq:-U}\\
&=
\big(\big\la \tilde{E}_{0,0}\big\ra,\big\la \tilde{E}_{0,1}\big\ra,\ldots,\big\la \tilde{E}_{0,n-1}\big\ra,\big\la \tilde{E}_{0,n}\big\ra\big)\tilde{L}\tilde{U},\label{eq:-L-U}
\end{align}
\end{subequations}
where the matrix $\tilde{U}=\big(\tilde{u}_{ij}\big)_{0\le i,j\le n}$ is defined by (\ref{eq:<En,j>c}).
Since $T_{\alpha}\Phi(z)=z_1z_2\cdots z_n\Phi(z)$
and $z_1z_2\cdots z_n \tilde{E}_{0,i}(z)=\tilde{E}_{n,i}(z)$,
we have $T_{\alpha}\big\la \tilde{E}_{0,i}\big\ra=\big\la \tilde{E}_{n,i}\big\ra$, i.e.,
\begin{equation}
\label{eq:T-a()}
T_{\alpha}\big(\big\la \tilde{E}_{0,0}\big\ra,\big\la \tilde{E}_{0,1}\big\ra,\ldots,\big\la \tilde{E}_{0,n-1}\big\ra,\big\la \tilde{E}_{0,n}\big\ra\big)
=\big(\big\la \tilde{E}_{n,0}\big\ra,\big\la \tilde{E}_{n,1}\big\ra,\ldots,\big\la \tilde{E}_{n,n-1}\big\ra,\big\la \tilde{E}_{n,n}\big\ra\big).
\end{equation}
From (\ref{eq:-L-U}) and (\ref{eq:T-a()}), we
obtain the difference system
\begin{equation*}
T_{\alpha}\big(\big\la \tilde{E}_{0,0}\big\ra,\big\la \tilde{E}_{0,1}\big\ra,\ldots,\big\la \tilde{E}_{0,n-1}\big\ra,\big\la \tilde{E}_{0,n}\big\ra\big)
=\big(\big\la \tilde{E}_{0,0}\big\ra,\big\la \tilde{E}_{0,1}\big\ra,\ldots,\big\la \tilde{E}_{0,n-1}\big\ra,\big\la \tilde{E}_{0,n}\big\ra\big)\tilde{L}\tilde{U}.
\end{equation*}
Comparing this with (\ref{eq:main}), we therefore obtain $A=\tilde{L}\tilde{U}=L_\tinyA D_\tinyA U_\tinyA$, i.e.,
\begin{equation*}
l^\tinyA_{ij}=\frac{\tilde{l}_{ij}}{\tilde{l}_{jj}},
\qquad
d^\tinyA_{j}=\tilde{l}_{jj}\tilde{u}_{jj},
\qquad
u^\tinyA_{ij}=\frac{\tilde{u}_{ij}}{\tilde{u}_{ii}}.
\end{equation*}
Corollary~\ref{cor:<Eki>} implies that $l^\tinyA_{ij}$, $d^\tinyA_{j}$ and $u^\tinyA_{ij}$ coincide with~\eqref{eq:lAij}, \eqref{eq:dAij} and~\eqref{eq:uAij}, respectively.

The Gauss decomposition of $A$ in the opposite direction, i.e.,
$A=U'{}_{\!\!\!\tinyA}D'{}_{\!\!\tinyA}L'{}_{\!\!\tinyA}$ can also be given in the same way as above.
This will be explained in the Appendix.
\end{proof}

Finally we state the explicit forms for $U_\tinyA^{-1}$.
\begin{Proposition}\label{prop:inverseU_A}
The inverse matrix $U_\tinyA^{-1}=\big(u^{\tinyA *}_{ij}\big)_{0\le i,j\le n}$ is upper triangular, and is written as
\begin{equation}\label{eq:inverseU_A}
u^{\tinyA *}_{ij}=
\big(q^\alpha a_1^{-1}a_2t^{j-1}\big)^{j-i}
\qbin{j}{i}{t}\frac{\big(a_1b_1 t^{n-j};t\big)_{j-i}}
{\big(q^\alpha a_2b_2 t^{j+i-1};t\big)_{j-i}}.
\end{equation}
\end{Proposition}

\begin{proof}Since $A=\tilde{L}\tilde{U}=L_\tinyA D_\tinyA U_\tinyA$, we have $\tilde{U}=D_{\tilde{U}}U_\tinyA$, where $D_{\tilde{U}}$ is the diagonal matrix
defined by the diagonal elements of $\tilde{U}=\big(\tilde{u}_{ij}\big)_{0\le i,j\le n}$, i.e.,
$
D_{\tilde{U}}=\big(\tilde{u}_{ii}\delta_{ij}\big)_{0\le i,j\le n}
$, where $\tilde{u}_{ij}$ is given by \eqref{eq:<En,j>c}.
We first compute the explicit expression for $\tilde{U}^{-1}$. From \eqref{eq:-U},
$\tilde{U}^{-1}$ is regarded as the transition matrix
\begin{gather}\label{E=EU^{-1}}
(\la E_{n,0}\ra,\la E_{n-1,1}\ra,\ldots,\la E_{1,n-1}\ra,\la E_{0,n}\ra)=
(\la E_{n,0}\ra,\la E_{n,1}\ra,\ldots,\la E_{n,n-1}\ra,\la E_{n,n}\ra)\tilde{U}^{-1},
\end{gather}
namely, if we write $\tilde{U}^{-1}=\big(\tilde{v}_{ij}\big)_{0\le i,j\le n}$, then \eqref{E=EU^{-1}} is equivalent to
$\big\la \tilde{E}_{n-j,j}\big\ra=\sum_{i=0}^j\tilde{v}_{ij}\big\la \tilde{E}_{n,i}\big\ra$.
Similar to Corollary \ref{cor:<Eki>}, by repeated use of the three-term relation \eqref{eq:<Eki>02} inductively, $\la\tilde{E}_{k,i}\ra$ is generally expressed as
\begin{equation*}%\label{eq:<Eki>03}
\big\la \tilde{E}_{k,i}\big\ra=\sum_{j=0}^l V_{k+l,i-j}^{k,i}\big\la \tilde{E}_{k+l,i-j}\big\ra,
\end{equation*}
where the coefficients $V_{k+l,i-j}^{k,i}$ is given by
\begin{equation*}
V_{k+l,i-j}^{k,i}
=
\qbin{l}{j}{t}
\frac{\big(q^\alpha a_1^{-1}t^{i-1}\big)^j t^{{l-j\choose 2}-{j\choose 2}}
\big(a_1b_1 t^{n-i};t\big)_{j}\big(q^\alpha a_2b_2 t^{n+i-k-l-1};t\big)_{l-j}}
{\big(a_2t^{n-k-1}\big)^{l-j}\big(q^\alpha t^{n-k-l};t\big)_{l}}.
\end{equation*}
In particular, the entries $\tilde{v}_{ij}$ of $\tilde{U}^{-1}$ are explicitly expressed as
\begin{equation}\label{eq:tilde{v}_{ij}}
\tilde{v}_{ij}=V_{n,i}^{n-j,j}=
\qbin{j}{i}{t}
\frac{\big(q^\alpha a_1^{-1}t^{j-1}\big)^{j-i} t^{{i\choose 2}-{j-i\choose 2}}
\big(a_1b_1 t^{n-j};t\big)_{j-i}\big(q^\alpha a_2b_2 t^{j-1};t\big)_{i}}
{\big(a_2t^{j-1}\big)^{i}\big(q^\alpha;t\big)_{j}}.
\end{equation}
Next we compute $U_\tinyA^{-1}=\big(u^{\tinyA *}_{ij}\big)_{0\le i,j\le n}$.
Since $U_\tinyA^{-1}$ is expressed as $U_\tinyA^{-1}=\tilde{U}^{-1}D_{\tilde{U}}$, using \eqref{eq:<En,j>c} and
\eqref{eq:tilde{v}_{ij}},
we obtain
\begin{gather*}
u^{\tinyA *}_{ij}=\tilde{v}_{ij}\tilde{u}_{jj}\\
\hphantom{u^{\tinyA *}_{ij}}{} =
\qbin{j}{i}{t}
\frac{\big(q^\alpha a_1^{-1}t^{j-1}\big)^{j-i} t^{{i\choose 2}-{j-i\choose 2}}
\big(a_1b_1 t^{n-j};t\big)_{j-i}\big(q^\alpha a_2b_2 t^{j-1};t\big)_{i}}
{\big(a_2t^{j-1}\big)^{i}\big(q^\alpha;t\big)_{j}}
\frac{
a_2^jt^{{j\choose 2}}
\big(q^\alpha;t\big)_{j}}
{\big(q^\alpha a_2b_2 t^{j-1};t\big)_{j}},
\end{gather*}
which coincides with \eqref{eq:inverseU_A}.
\end{proof}

\section{Proof of Lemma \ref{lem:3term1st}}\label{section05}

The aim of this section is to give a proof of Lemma \ref{lem:3term1st}.
Throughout this section we fix $\tilde{E}_{k,i}(z)=\tilde{E}_{k,i}(a_1, b_2;z)$.
For $\Phi(z)=\Phi_{n,2}$, let $\nabla$ be the operator specified by
\begin{equation*}
(\nabla\varphi)(z):=\varphi(z)-\frac{T_{q,z_1}\Phi(z)}{\Phi(z)}T_{q,z_1}\varphi(z),
\end{equation*}
where $T_{q,z_1}$ is the $q$-shift operator with respect to $z_1\to qz_1$, i.e., $T_{q,z_1}
f(z_1,z_2,\ldots,z_n)=f(qz_1,z_2,\ldots,z_n)$ for an arbitrary function $f(z_1,z_2,\ldots,z_n)$.
Here the ratio $T_{q,z_1}\Phi(z)/\Phi(z)$ is expressed explicitly as
\begin{equation*}
\frac{T_{q,z_1}\Phi(z)}{\Phi(z)}=
\frac{q^\alpha t^{2(n-1)}(1-b_1z_1)
(1-b_2z_1)}{\big(1-qa_1^{-1}z_1\big)\big(1-qa_2^{-1}z_1\big)}
\prod_{j=2}^n
\frac{z_1-t^{-1}z_j}{qz_1-tz_j}
=\frac{G_1(z)}{T_{z_1}F_1(z)},
\end{equation*}
where
\begin{gather*}
F_1(z) =\big(1-a_1^{-1}z_1\big)\big(1-a_2^{-1}z_1\big)\prod_{j=2}^n(z_1-tz_j),\\
G_1(z) =q^\alpha t^{2(n-1)}(1-b_1z_1)(1-b_2z_1)\prod_{j=2}^n\big(z_1-t^{-1}z_j\big).
\end{gather*}

\begin{Lemma}\label{lem:nabla=0}Suppose that
$\int_0^{x\infty}\Phi(z)\varphi(z)\varpi_q$
converges for a meromorphic function $\varphi(z)$, then
\begin{equation*}
\int_0^{x\infty}\Phi(z)\nabla\varphi(z)\varpi_q=0.
\end{equation*}
Moreover,
\begin{equation*}
\int_0^{x\infty}\Phi(z){\cal A}\nabla\varphi(z)\varpi_q=0.
\end{equation*}
\end{Lemma}

\begin{proof}See Lemma~5.3 in \cite{IN2018}.
\end{proof}

The rest of this section is devoted to the proof of Lemma~\ref{lem:3term1st}.
We show a further lemma before proving Lemma~\ref{lem:3term1st}.
For this purpose we abbreviate $\tilde{E}_{k,i}(a_1,b_2;z)$ to $\tilde{E}_{k,i}(z)$.
When we need to specify the number of variables $z_1,\ldots,z_n$, we use the notation
$\tilde{E}_{k,i}^{(n)}(z)=\tilde{E}_{k,i}(z)$ and $\Delta^{\!(n)}(z)=\Delta(z)$.
We set $\varphi_{k,i}(z):=F_1(z)E_{k-1,i}^{(n-1)}(z_2,\ldots,z_n)$. Then
\begin{equation*}
\nabla\varphi_{k,i}(z)= (F_1(z)-G_1(z))E_{k-1,i}^{(n-1)}(z_2,\ldots,z_n).
\end{equation*}
Let $\tilde{\varphi}_{k,i}(z)$ be the skew-symmetrization of $\nabla\varphi_{k,i}(z)$, i.e.,
\begin{equation}
\label{eq:tildevarphi_k,i}
\tilde{\varphi}_{k,i}(z):={\cal A}\nabla\varphi_{k,i}(z)=\sum_{j=1}^n(-1)^{j-1} (F_j(z)-G_j(z))
\tilde{E}_{k-1,i}^{(n-1)}(\widehat{z}_{j})\Delta^{\!(n-1)}(\widehat{z}_{j}),
\end{equation}
where $\widehat{z}_{j}:=(z_1,\ldots,z_{j-1},z_{j+1},\ldots,z_n)\in (\mathbb{C}^*)^{n-1}$ for $j=1,\ldots,n$, and
\begin{gather*}
F_i(z)=\big(1-a_1^{-1}z_i\big)\big(1-a_2^{-1}z_i\big)
\prod_{\substack{1\le k\le n\\[1pt] k\ne i}
}(z_i-tz_k),\\
G_i(z)=q^\alpha t^{2(n-1)}(1-b_1z_i)
(1-b_2z_i)
\prod_{\substack{1\le k\le n\\[1pt] k\ne i}}\big(z_i-t^{-1}z_k\big),
\end{gather*}
which satisfy the following vanishing property at the point $z=\zeta_j\big(x,b_2^{-1}\big)$ or $z=\zeta_j(a_1,y)$.

\begin{Lemma}\label{lem:vanishingFG}If $i\ne 0$ and $i\ne j$, then $F_{i+1}\big(\zeta_j\big(x,b_2^{-1}\big)\big)=0$.
If $i\ne n$, then
$G_i\big(\zeta_j\big(x,b_2^{-1}\big)\big)=0$.
Otherwise,
\begin{subequations}
\begin{gather}
F_{1}\big(\zeta_j\big(x,b_2^{-1}\big)\big)=
\frac{\big(1-a_1^{-1}b_2^{-1}t^{-(j-1)}\big)\big(1-a_2^{-1}b_2^{-1}t^{-(j-1)}\big)}{\big(b_2t^{j-1}\big)^{n-1}(1-t)}
(t;t)_j\big(xb_2t^{j};t\big)_{n-j},
\label{eq:vanishingFG-1}\\
F_{j+1}\big(\zeta_j\big(x,b_2^{-1}\big)\big)=(-1)^j\frac{\big(1-a_1^{-1}x\big)\big(1-a_2^{-1}x\big)x^{n-j-1}}
{b_2^jt^{{j-1\choose 2}-1}(1-t)}(t,t)_{n-j}\big(xb_2t^{-1};t\big)_{j},
\label{eq:vanishingFG-2}\\
G_n\big(\zeta_j\big(x,b_2^{-1}\big)\big)=
q^\alpha t^{2(n-1)}\big(1-b_1xt^{n-j-1}\big)\big(1-b_2xt^{n-j-1}\big)\nonumber\\
\hphantom{G_n\big(\zeta_j\big(x,b_2^{-1}\big)\big)=}{}
\times\big({-}b_2^{-1}\big)^jt^{-{j+1\choose 2}}\big(xbt^{n-j};t\big)_j\big(xt^{n-j-1}\big)^{n-j-1}\frac{\big(t^{-1};t^{-1}\big)_{n-j}}{1-t^{-1}},
\label{eq:vanishingFG-3}
\end{gather}
\end{subequations}
while, if $i\ne 1$, then $F_i(\zeta_j(a_1,y))=0$. If $i\ne n$ and $i\ne j$, then $G_i(\zeta_j(a_1,y))=0$. Otherwise,
\begin{subequations}
\begin{gather}
F_1(\zeta_j(a_1,y)) =\big(1-ya_2^{-1}t^{-(j-1)}\big)(-yt)^{j-1}(-a_1t)^{n-j}t^{{n-j\choose 2}-{j-1\choose 2}}
\nonumber\\
\hphantom{F_1(\zeta_j(a_1,y)) =}{}
\times \big(ya_1^{-1}t^{-(n-1)};t\big)_{n-j+1}
\frac{\big(t^{-1};t^{-1}\big)_j}{1-t^{-1}},
\label{eq:vanishingFG-4}\\
G_j(\zeta_j(a_1,y)) =q^\alpha t^{2(n-1)}(1-yb_1)(1-yb_2)
\nonumber\\
\hphantom{G_j(\zeta_j(a_1,y)) =}{}
\times y^{j-1}(-a_1t^{-1})^{n-j}t^{{n-j\choose 2}}
\big(ya_1^{-1}t^{-(n-j-2)};t\big)_{n-j}\frac{\big(t^{-1};t^{-1}\big)_j}{1-t^{-1}},
\label{eq:vanishingFG-5}\\
G_n(\zeta_j(a_1,y)) =q^\alpha t^{2(n-1)}\big(1-a_1b_1t^{n-j-1}\big)\big(1-a_1b_2t^{n-j-1}\big)
\nonumber\\
\hphantom{G_n(\zeta_j(a_1,y)) =}{}
\times\big(a_1t^{n-j-1}\big)^{n-1}\big(ya_1^{-1}t^{-(n-1)};t\big)_j
\frac{\big(t^{-1};t^{-1}\big)_{n-j}}{1-t^{-1}}.\label{eq:vanishingFG-6}
\end{gather}
\end{subequations}
\end{Lemma}

\begin{proof}
The proof follows by direct computation and we omit the details.
\end{proof}

Since the leading term of the symmetric polynomial $\tilde{\varphi}_{k,i}(z)/\Delta^{\!(n)}(z)$ is equal to
$m_{(1^{n-k}2^k)}(z)$ up to a multiplicative constant,
$\tilde{\varphi}_{k,i}(z)/\Delta^{\!(n)}(z)$ is
expressed as the linear combination of the symmetric polynomials $\tilde{E}_{l,j}^{(n)}(z)$ in the following two ways:
\begin{equation}\label{eq:expand02}
\frac{\tilde{\varphi}_{k,i}(z)}{\Delta^{\!(n)}(z)}=
\sum_{l=0}^{k}\sum_{j=0}^{n-l}c_{lj}\tilde{E}_{l,j}^{(n)}(z)
=\sum_{l=0}^{k}\sum_{j=n-l}^{n}d_{lj}\tilde{E}_{l,j}^{(n)}(z),
\end{equation}
where $c_{lj}$ and $d_{lj}$ are some coefficients.

\begin{Lemma}\label{lem:3term1st-01} Suppose $i+k\le n$. Then,
\eqref{eq:expand02} is written as
\begin{equation}\label{eq:expand03}
\frac{\tilde{\varphi}_{k,i}(z)}{\Delta^{\!(n)}(z)}
=c_{k,i}\tilde{E}_{k,i}^{(n)}(z)+c_{k-1,i}\tilde{E}_{k-1,i}^{(n)}(z)+c_{k-1,i+1}\tilde{E}_{k-1,i+1}^{(n)}(z),
\end{equation}
where
\begin{subequations}
\begin{gather}
c_{k,i}=-a_1^{-1}a_2^{-1}b_2^{-1}t^{k-1}\big(1-q^\alpha a_1 a_2 b_1 b_2 t^{2n-k-1}\big)
=q^\alpha b_1 t^{2n-2}-a_1^{-1}a_2^{-1}b_2^{-1}t^{k-1},
\label{eq:ck,i}\\
c_{k-1,i}=a_2^{-1}b_2^{-1}t^{n-i-1}\big(1-q^\alpha a_2 b_2 t^{n+i-k}\big)
=a_2^{-1}b_2^{-1}t^{n-i-1}-q^\alpha t^{2n-k-1},
\label{eq:ck-1,i}\\
c_{k-1,i+1}=-a_2^{-1}b_2^{-1}t^{n-i-1}\big(1-a_2 b_2 t^{i}\big)
=t^{n-1}\big(1-a_2^{-1}b_2^{-1}t^{-i}\big).
\label{eq:ck-1,i-1}
\end{gather}
\end{subequations}
Suppose $i+k\ge n$. Then, \eqref{eq:expand02} is written as
\begin{equation}\label{eq:expand04}
\frac{\tilde{\varphi}_{k,i}(z)}{\Delta^{\!(n)}(z)}
=d_{k,i}\tilde{E}_{k,i}^{(n)}(z)+d_{k,i+1}\tilde{E}_{k,i+1}^{(n)}(z)+d_{k-1,i+1}\tilde{E}_{k-1,i+1}^{(n)}(z),
\end{equation}
where
\begin{subequations}
\begin{gather}
d_{k,i} =-a_1^{-1}q^\alpha t^{n+i-1}\big(1- a_1 b_1 t^{n-i-1}\big)=q^\alpha b_1 t^{2n-2}-q^\alpha a_1^{-1}t^{n+i-1},\label{eq:dk,i}\\
d_{k,i+1} =-a_2^{-1}t^{k-1}\big(1-q^\alpha a_2 b_2 t^{n+i-k}\big)=q^\alpha b_2t^{n+i-1}-a_2^{-1}t^{k-1},\label{eq:dk,i+1}\\
d_{k-1,i+1} =t^{n-1}\big(1-q^\alpha t^{n-k}\big)=t^{n-1}-q^\alpha t^{2n-k-1}.\label{eq:dk-1,i+1}
\end{gather}
\end{subequations}
\end{Lemma}

\begin{Remark}Given Lemma \ref{lem:3term1st-01}, Lemma \ref{lem:3term1st} immediately follows by
Lemma \ref{lem:nabla=0}.
Instead of proving Lemma \ref{lem:3term1st} it thus suffices to prove Lemma \ref{lem:3term1st-01}.
\end{Remark}

Before proving Lemma \ref{lem:3term1st-01} we show it holds for the following specific cases.

\begin{Lemma}\label{lem:3term1st-02}
If $i+k\le n$, then the equation \eqref{eq:expand03} holds for the points $z=\zeta_j\big(x,b_2^{-1}\big)$ $(j=0,1,\ldots,n)$,
while if $i+k\ge n$, then the equation~\eqref{eq:expand04} holds
for the points $z=\zeta_j(a_1,y)$ $(j=0,1,\ldots,n)$.
\end{Lemma}

\begin{proof}Suppose $i+k\le n$. If $z=\zeta_j\big(x,b_2^{-1}\big)$, then the right-hand side of~(\ref{eq:expand03}) with
coefficients given by (\ref{eq:ck,i})--(\ref{eq:ck-1,i-1}) can be written as
\begin{gather}
 c_{k,i}\tilde{E}_{k,i}^{(n)}\big(\zeta_j\big(x,b_2^{-1}\big)\big)+c_{k-1,i}\tilde{E}_{k-1,i}^{(n)}\big(\zeta_j\big(x,b_2^{-1}\big)\big)
+c_{k-1,i+1}\tilde{E}_{k-1,i+1}^{(n)}\big(\zeta_j\big(x,b_2^{-1}\big)\big)
\nonumber\\
\qquad{}=\big(c_{k,i}xt^{n-j-k}+c_{k-1,i}\big)
\tilde{E}_{k-1,i}^{(n)}\big(\zeta_j\big(x,b_2^{-1}\big)\big)
+ c_{k-1,i+1}\tilde{E}_{k-1,i+1}^{(n)}\big(\zeta_j\big(x,b_2^{-1}\big)\big)\nonumber\\
\qquad{}=\big[\big(1-a_1^{-1}xt^{i-j}\big)a_2^{-1}b_2^{-1}t^{n-i-1}-q^\alpha t^{2n-k-1}\big(1-xb_1t^{n-j-1}\big)\big]
\tilde{E}_{k-1,i}^{(n)}\big(\zeta_j\big(x,b_2^{-1}\big)\big)\nonumber\\[2pt]
\qquad\quad {}+ t^{n-1}\big(1-a_2^{-1}b_2^{-1}t^{-i}\big)\tilde{E}_{k-1,i+1}^{(n)}\big(\zeta_j\big(x,b_2^{-1}\big)\big)\nonumber\\
\qquad{}=\left[\frac{\big(1-xb_2 t^i\big)\big(1-t^{i-j+1}\big)}{a_2b_2 t^i\big(1-t^{i+1}\big)
}+\big(1-a_2^{-1}b_2^{-1}t^{-i}\big)\right]t^{n-1}\tilde{E}_{k-1,i+1}^{(n)}\big(\zeta_j\big(x,b_2^{-1}\big)\big)\nonumber\\
\qquad\quad {}-q^\alpha t^{2n-k-1}\big(1-xb_1t^{n-j-1}\big) \tilde{E}_{k-1,i}^{(n)}\big(\zeta_j\big(x,b_2^{-1}\big)\big).\label{eq:expand05}
\end{gather}
The final equality follows from the relation
\begin{gather*}
\big(1-a_1^{-1}xt^{i-j}\big)\big(1-t^{i+1}\big)\tilde{E}_{k-1,i}^{(n)}\big(\zeta_j\big(x,b_2^{-1}\big)\big)=
\big(1-xb_2 t^i\big)\big(1-t^{i-j+1}\big)\tilde{E}_{k-1,i+1}^{(n)}\big(\zeta_j\big(x,b_2^{-1}\big)\big),
\end{gather*}
which follows from (\ref{eq:tri-zeta(xb)2}) and (\ref{eq:tri-zeta(xb)3}).
On the other hand, using
\eqref{eq:tildevarphi_k,i} and
\eqref{eq:vanishingFG-1}--\eqref{eq:vanishingFG-3},
the left-hand side of~(\ref{eq:expand03}) at $z=\zeta_j\big(x,b_2^{-1}\big)$
can be written as
\begin{gather}
 \frac{\tilde{\varphi}_{k,i}\big(\zeta_{j}\big(x,b^{-1}\big)\big)}{\Delta^{\!(n)}\big(\zeta_{j}^{(n)}\big(x,b^{-1}\big)\big)}
=F_1\big(\zeta_{j}^{(n)}\big(x,b^{-1}\big)\big)\tilde{E}_{k-1,i}^{(n-1)}\big(\zeta_{j-1}^{(n-1)}\big(x,b^{-1}\big)\big)
\frac{\Delta^{\!(n-1)}\big(\zeta_{j-1}^{(n-1)}\big(x,b^{-1}\big)\big)}
{\Delta^{\!(n)}\big(\zeta_{j}^{(n)}\big(x,b^{-1}\big)\big)}
\nonumber\\
\qquad {}+(-1)^{j}F_{j+1}\big(\zeta_{j}^{(n)}\big(x,b^{-1}\big)\big)\tilde{E}_{k-1,i}^{(n-1)}\big(\zeta_{j}^{(n-1)}\big(xt,b^{-1}\big)\big)
\frac{\Delta^{\!(n-1)}\big(\zeta_{j}^{(n-1)}\big(xt,b^{-1}\big)\big)}
{\Delta^{\!(n)}\big(\zeta_{j}^{(n)}\big(x,b^{-1}\big)\big)}
\nonumber\\
\qquad {}+(-1)^{n}G_{n}\big(\zeta_{j}^{(n)}\big(x,b^{-1}\big)\big)\tilde{E}_{k-1,i}^{(n-1)}\big(\zeta_{j}^{(n-1)}\big(x,b^{-1}\big)\big)
\frac{\Delta^{\!(n-1)}\big(\zeta_{j}^{(n-1)}\big(x,b^{-1}\big)\big)}
{\Delta^{\!(n)}\big(\zeta_{j}^{(n)}\big(x,b^{-1}\big)\big)}.\label{eq:expand06}
\end{gather}
Since we can compute
\begin{subequations}
\begin{gather}
\tilde{E}_{k-1,i}^{(n-1)}\big(\zeta_{j-1}^{(n-1)}\big(x,b^{-1}\big)\big)=
\frac{t^{n-1}(1-t)(1-xbt^i)\tilde{E}_{k-1,i+1}^{(n)}\big(\zeta_{j}^{(n)}\big(x,b^{-1}\big)\big)}{\big(1-xbt^{n-1}\big) \big(1-a^{-1}b^{-1}t^{-(j-1)}\big)\big(1-t^{i+1}\big)},
\label{eq:E-E(zeta(xb))1}\\
\tilde{E}_{k-1,i}^{(n-1)}\big(\zeta_{j}^{(n-1)}\big(xt,b^{-1}\big)\big)=
\frac{t^{n-1}(1-t)(1-t^{i-j+1})\tilde{E}_{k-1,i+1}^{(n)}\big(\zeta_{j}^{(n)}\big(x,b^{-1}\big)\big)}{\big(1-xa^{-1}\big)\big(1-t^{i+1}\big) \big(1-t^{n-j}\big)},
\label{eq:E-E(zeta(xb))2}\\
\tilde{E}_{k-1,i}^{(n-1)}\big(\zeta_{j}^{(n-1)}\big(x,b^{-1}\big)\big)
=\frac{t^{n-k}(1-t)\tilde{E}_{k-1,i}^{(n)}\big(\zeta_{j}^{(n)}\big(x,b^{-1}\big)\big)}{\big(1-xbt^{n-1}\big)\big(1-t^{n-j}\big)},
\label{eq:E-E(zeta(xb))3}
\end{gather}
and
\begin{gather}
\frac{\Delta^{\!(n-1)}\big(\zeta_{j-1}^{(n-1)}\big(x,b^{-1}\big)\big)}{\Delta^{\!(n)}\big(\zeta_{j}^{(n)}\big(x,b^{-1}\big)\big)}
=\frac{\big(b_2t^{j-1}\big)^{n-1}}{(t;t)_{j-1}\big(xbt^{j-1};t\big)_{n-j}},
\label{eq:D-D(zeta(xb))1}\\
\frac{\Delta^{\!(n-1)}\big(\zeta_{j}^{(n-1)}\big(xt,b^{-1}\big)\big)}{\Delta^{\!(n)}\big(\zeta_{j}^{(n)}\big(x,b^{-1}\big)\big)}
=\frac{x^{-(n-j-1)}b_2^jt^{{j\choose 2}}}{(t;t)_{n-j-1}(xb;t)_j}
\label{eq:D-D(zeta(xb))2}\\
\frac{\Delta^{\!(n-1)}\big(\zeta_{j}^{(n-1)}\big(xb^{-1}\big)\big)}{\Delta^{\!(n)}\big(\zeta_{j}^{(n)}\big(x,b^{-1}\big)\big)}
=\frac{x^{-(n-j-1)}b_2^jt^{{j\choose 2}-{n-j-1\choose 2}}}{(t;t)_{n-j-1}\big(xbt^{n-j-1};t\big)_{j}},
\label{eq:D-D(zeta(xb))3}
\end{gather}
\end{subequations}
applying (\ref{eq:E-E(zeta(xb))1})--(\ref{eq:D-D(zeta(xb))3}) to (\ref{eq:expand06}),
the left-hand side of~(\ref{eq:expand03}) at $z=\zeta_{j}\big(x,b_2^{-1}\big)$
can be expressed as
\begin{gather}
\frac{\tilde{\varphi}_{k,i}\big(\zeta_{j}\big(x,b_2^{-1}\big)\big)}{\Delta^{\!(n)}\big(\zeta_{j}^{(n)}\big(x,b_2^{-1}\big)\big)}
 =\bigg[\frac{\big(1-a_2^{-1}b_2^{-1}t^{-(j-1)}\big)\big(1-t^j\big)\big(1-xb_2t^{i}\big)}{\big(1-xb_2t^{j-1}\big)\big(1-t^{i+1}\big)}\nonumber\\
\hphantom{\frac{\tilde{\varphi}_{k,i}\big(\zeta_{j}\big(x,b_2^{-1}\big)\big)}{\Delta^{\!(n)}\big(\zeta_{j}^{(n)}\big(x,b_2^{-1}\big)\big)=}}{}
+\frac{\big(1-a_2^{-1}x\big)\big(1-xb_2t^{-1}\big)\big(1-t^{i-j+1}\big)}{\big(1-xb_2t^{j-1}\big)\big(1-t^{i+1}\big)}t^{j}\bigg]
t^{n-1}\tilde{E}_{k-1,i+1}^{(n)}\big(\zeta_{j}^{(n)}\big(x,b_2^{-1}\big)\big)
\nonumber\\
\hphantom{\frac{\tilde{\varphi}_{k,i}\big(\zeta_{j}\big(x,b_2^{-1}\big)\big)}{\Delta^{\!(n)}\big(\zeta_{j}^{(n)}\big(x,b_2^{-1}\big)\big)=}}{}
-q^\alpha t^{2n-k-1}\big(1-xb_1t^{n-j-1}\big)\tilde{E}_{k-1,i}^{(n)}\big(\zeta_{j}^{(n)}\big(x,b_2^{-1}\big)\big).\label{eq:expand07}
\end{gather}
Comparing (\ref{eq:expand05})
with (\ref{eq:expand07}), the claim of the lemma is proved if we can check the identity
\begin{gather*}
 \frac{\big(1-xb_2 t^i\big)\big(1-t^{i-j+1}\big)}{a_2b_2 t^i(1-t^{i+1})
}+\big(1-a_2^{-1}b_2^{-1}t^{-i}\big)\\
 \qquad{}=
\frac{\big(1-a_2^{-1}b_2^{-1}t^{-(j-1)}\big)\big(1-t^j\big)\big(1-xb_2t^{i}\big)}{\big(1-xb_2t^{j-1}\big)\big(1-t^{i+1}\big)}
+\frac{\big(1-a_2^{-1}x\big)\big(1-xb_2t^{-1}\big)\big(1-t^{i-j+1}\big)}{\big(1-xb_2t^{j-1}\big)\big(1-t^{i+1}\big)}t^{j},
\end{gather*}
which is confirmed by direct computation.

Next suppose $i+k\ge n$. If $z=\zeta_j(a_1,y)$,
then the right-hand side of (\ref{eq:expand04}) with coefficients given by (\ref{eq:dk,i})--(\ref{eq:dk-1,i+1})
can be written as
\begin{gather}
d_{k,i}\tilde{E}_{k,i}^{(n)}(\zeta_j(a_1,y)) +d_{k,i+1}\tilde{E}_{k,i+1}^{(n)}(\zeta_j(a_1,y))+d_{k-1,i+1}\tilde{E}_{k-1,i+1}^{(n)}(\zeta_j(a_1,y))
\nonumber\\
{}=d_{k,i}\tilde{E}_{k,i}^{(n)}(\zeta_j(a_1,y))+\big[ d_{k,i+1}yt^{-(k+j-n-1)}+d_{k-1,i+1}\big]\tilde{E}_{k-1,i+1}^{(n)}(\zeta_j(a_1,y))
\nonumber\\
{}=-q^\alpha a_1^{-1}t^{n+i-1}\big(1-a_1b_1t^{n-i-1}\big)\tilde{E}_{k,i}^{(n)}(\zeta_j(a_1,y))
\nonumber\\
\quad{} +\big[t^{n-1}\big(1-ya_2^{-1}t^{-(j-1)}\big)-q^\alpha t^{2n-k-1}\big(1-yb_2 t^{-(j-i-1)}\big)\big]\tilde{E}_{k-1,i+1}^{(n)}(\zeta_j(a_1,y))
\nonumber\\
{}=
-q^\alpha t^{n+i-1}\left[ a_1^{-1}\big(1-a_1b_1t^{n-i-1}\big)
+\frac{\big(1-ya_1^{-1}t^{-(n-i-1)}\big)\big(1-t^{-(j-i)}\big)}{yt^{-(j-i-1)}\big(1-t^{-(n-i)}\big)}\right] \tilde{E}_{k,i}^{(n)}(\zeta_j(a_1,y))
\nonumber\\
\quad{} +t^{n-1}\big(1-ya_2^{-1}t^{-(j-1)}\big)\tilde{E}_{k-1,i+1}^{(n)}(\zeta_j(a_1,y))
\nonumber\\
{}=
-q^\alpha t^{n+j-2}\left[ \frac{1-a_1b_1t^{n-i-1}}{a_1t^{j-i-1}}
+\frac{\big(1-ya_1^{-1}t^{-(n-i-1)}\big)\big(1-t^{-(j-i)}\big)}{y\big(1-t^{-(n-i)}\big)}\right]\tilde{E}_{k,i}^{(n)}(\zeta_j(a_1,y))
\nonumber\\
\quad{} +t^{n-1}\big(1-ya_2^{-1}t^{-(j-1)}\big)\tilde{E}_{k-1,i+1}^{(n)}(\zeta_j(a_1,y)).
\label{eq:expand08}
\end{gather}
On the other hand, using \eqref{eq:tildevarphi_k,i} and
\eqref{eq:vanishingFG-4}--\eqref{eq:vanishingFG-6},
the left-hand side of (\ref{eq:expand04}) at $z=\zeta_j(a_1,y)$
can be written as
\begin{gather}
 \frac{\tilde{\varphi}_{k,i}(\zeta_{j}(a_1,y))}{\Delta^{\!(n)}\big(\zeta_{j}^{(n)}(a_1,y)\big)}
=F_1(\zeta_{j}^{(n)}(a_1,y))\tilde{E}_{k-1,i}^{(n-1)}\big(\zeta_{j-1}^{(n-1)}(a_1,y)\big)
\frac{\Delta^{\!(n-1)}\big(\zeta_{j-1}^{(n-1)}(a_1,y)\big)}
{\Delta^{\!(n)}\big(\zeta_{j}^{(n)}(a_1,y)\big)}
\nonumber\\
\qquad{} -(-1)^{j-1}G_{j}\big(\zeta_{j}^{(n)}(a_1,y)\big)\tilde{E}_{k-1,i}^{(n-1)}\big(\zeta_{j-1}^{(n-1)}\big(a_1,yt^{-1}\big)\big)
\frac{\Delta^{\!(n-1)}\big(\zeta_{j-1}^{(n-1)}\big(a_1,yt^{-1}\big)\big)}
{\Delta^{\!(n)}\big(\zeta_{j}^{(n)}(a_1,y)\big)}
\nonumber\\
\qquad {}-(-1)^{n-1}G_{n}\big(\zeta_{j}^{(n)}(a_1,y)\big)\tilde{E}_{k-1,i}^{(n-1)}\big(\zeta_{j}^{(n-1)}(a_1,y)\big)
\frac{\Delta^{\!(n-1)}\big(\zeta_{j}^{(n-1)}(a_1,y)\big)}
{\Delta^{\!(n)}\big(\zeta_{j}^{(n)}(a_1,y)\big)}.
\label{eq:expand09}
\end{gather}
Since we can compute
\begin{subequations}
\begin{gather}
\tilde{E}_{k-1,i}^{(n-1)}\big(\zeta_{j-1}^{(n-1)}(a,y)\big)=
\frac{\big(1-t^{-1}\big)\tilde{E}_{k-1,i+1}^{(n)}\big(\zeta_{j}^{(n)}(a,y)\big)}
{\big(1-t^{-j}\big)\big(1-ya^{-1}t^{-(n-1)}\big)},
\label{eq:E-E(zeta(ay))1}\\
\tilde{E}_{k-1,i}^{(n-1)}\big(\zeta_{j-1}^{(n-1)}\big(a,yt^{-1}\big)\big)=
\frac{\big(1-t^{-1}\big)\big(1-t^{-(j-i)}\big)\tilde{E}_{k,i}^{(n)}\big(\zeta_{j}^{(n)}(a,y)\big)}{\big(1-t^{-j}\big)\big(1-t^{-(n-i)}\big)(1-yb)},
\\
\tilde{E}_{k-1,i}^{(n-1)}\big(\zeta_{j}^{(n-1)}(a,y)\big)
=\frac{\big(1-t^{-1}\big)\big(1-ya^{-1}t^{-(n-i-1)}\big)\tilde{E}_{k,i}^{(n)}\big(\zeta_{j}^{(n)}(a,y)\big)}
{at^{n-j-1}\big(1-t^{-(n-i)}\big)\big(1-ya^{-1}t^{-(n-1)}\big)\big(1-abt^{n-j-1}\big)},
\end{gather}
and
\begin{gather}
\frac{\Delta^{\!(n-1)}\big(\zeta_{j-1}^{(n-1)}(a_1,y)\big)}{\Delta^{\!(n)}\big(\zeta_{j}^{(n)}(a_1,y)\big)}
 =\frac{(-1)^{n-j}}
{y^{j-1}a_1^{n-j}t^{{n-j\choose 2}-{j-1\choose 2}}\big(ya_1^{-1}t^{-(n-2)};t\big)_{n-j}\big(t^{-1};t^{-1}\big)_{j-1}},\\
\frac{\Delta^{\!(n-1)}\big(\zeta_{j-1}^{(n-1)}(a_1,yt^{-1})\big)}{\Delta^{\!(n)}\big(\zeta_{j}^{(n)}(a_1,y)\big)}
 =\frac{(-1)^{n-1}}
{y^{j-1}a_1^{n-j}t^{{n-j\choose 2}}\big(ya_1^{-1}t^{-(n-j-1)};t\big)_{n-j}\big(t^{-1};t^{-1}\big)_{j-1}},\\
\frac{\Delta^{\!(n-1)}\big(\zeta_{j}^{(n-1)}(a_1,y)\big)}
{\Delta^{\!(n)}\big(\zeta_{j}^{(n)}(a_1,y)\big)} =\frac{(-1)^{n-1}}{\big(a_1t^{n-j-1}\big)^{n-1}\big(ya_1^{-1}t^{-(n-2)};t\big)_j\big(t^{-1};t^{-1}\big)_{n-j-1}},
\label{eq:D-E(zeta(ay))3}
\end{gather}
\end{subequations}
applying (\ref{eq:E-E(zeta(ay))1})--(\ref{eq:D-E(zeta(ay))3}) to (\ref{eq:expand09}),
the left-hand side of~(\ref{eq:expand04}) at $z=\zeta_{j}(a,y)$
can be expressed as
\begin{gather}
 \frac{\tilde{\varphi}_{k,i}(\zeta_{j}(a_1,y))}{\Delta^{\!(n)}\big(\zeta_{j}^{(n)}(a_1,y)\big)}
=t^{n-1}\big(1-ya_2^{-1}t^{-(j-1)}\big)\tilde{E}_{k-1,i+1}^{(n)}(\zeta_j(a_1,y))\nonumber\\
\qquad{}-q^\alpha t^{n+j-2}\bigg[\frac{(1-yb_1)\big(1-ya^{-1}t\big)\big(1-t^{-(j-i)}\big)}{y\big(1-ya^{-1}t^{-(n-j-1)}\big)\big(1-t^{-(n-i)}\big)}\nonumber\\
\qquad{}+\big(1-a_1b_1 t^{n-j-1}\big)\frac{t\big(1-ya^{-1}t^{-(n-i-1)}\big) \big(1-t^{-(n-j)}\big)}{a\big(1-ya^{-1}t^{-(n-j-1)}\big)\big(1-t^{-(n-i)}\big)}\bigg]
\tilde{E}_{k,i}^{(n)}(\zeta_j(a_1,y)).
\label{eq:expand10}
\end{gather}
Comparing with (\ref{eq:expand08}) and (\ref{eq:expand10}),
the claim of the lemma is proved if we can check the identity
\begin{gather*}
 \frac{1-a_1b_1t^{n-i-1}}{a_1t^{j-i-1}}
+\frac{\big(1-ya_1^{-1}t^{-(n-i-1)}\big)\big(1-t^{-(j-i)}\big)}{y\big(1-t^{-(n-i)}\big)}
=\frac{(1-yb_1)\big(1-ya^{-1}t\big)\big(1-t^{-(j-i)}\big)}{y\big(1-ya^{-1}t^{-(n-j-1)}\big)\big(1-t^{-(n-i)}\big)}
\\
\qquad{}+\big(1-a_1b_1 t^{n-j-1}\big)\frac{t\big(1-ya^{-1}t^{-(n-i-1)}\big)\big(1-t^{-(n-j)}\big)}{a\big(1-ya^{-1}t^{-(n-j-1)}\big)\big(1-t^{-(n-i)}\big)},
\end{gather*}
which follows from direct computation.
\end{proof}

\begin{proof}[Proof of Lemma \ref{lem:3term1st-01}]
Set $D_j=\big\{(l,i)\in \mathbb{Z}^2\,|\, j\le i, 0\le l, i+l\le n\big\}$, which satisfies $D_0\supset D_1\supset\cdots\supset D_n=\{(0,n)\}$.
The set $\big\{\tilde{E}_{k,i}(z)\,|\, (k,i)\in D_0\big\}$ forms a basis for the linear space spanned
by $\{m_\lambda(z)\,|\, \lambda\le (2^n)\}$.
If we put
\begin{equation*}
\psi(z):=
\frac{\tilde{\varphi}_{k,i}(z)}{\Delta(z)}
-\big(c_{k,i}\tilde{E}_{k,i}(z)+c_{k-1,i}\tilde{E}_{k-1,i}(z)+c_{k-1,i+1}\tilde{E}_{k-1,i+1}(z)\big),
\end{equation*}
where $c_{k,i}$, $c_{k-1,i}$ and $c_{k-1,i+1}$ are specified by (\ref{eq:ck,i})--(\ref{eq:ck-1,i-1}),
then the symmetric polynomial $\psi(z)$ is expressed as a linear combination of $\tilde{E}_{k,i}(z)$, $(k,i)\in D_0$, i.e.,
\begin{equation}\label{eq:psi(z)}
\psi(z)=\sum_{(l,m)\in D_0}c'_{lm}\tilde{E}_{l,m}(z),
\end{equation}
where the coefficients $c'_{lm}$ are some constants. We now prove $\psi(z)=0$ identically, i.e.,
$c'_{lm}=0$ for all $(l,m)\in D_0$ inductively.
Namely, we prove that, if $c'_{lm}=0$ for $(l,m)\in D_{j+1}$, then $c'_{lm}=0$ for $(l,m)\in D_{j}.$

First we show that $c'_{0n}=0$ as the starting point of induction.
Using Lemma~\ref{lem:tri-zeta(xb)} for (\ref{eq:psi(z)}) at $z=\zeta_n\big(x,b_2^{-1}\big)$ we have
$\psi\big(\zeta_n\big(x,b_2^{-1}\big)\big)=c'_{0n}\tilde{E}_{0,n}\big(\zeta_n\big(x,b_2^{-1}\big)\big)$.
From Lemma~\ref{lem:3term1st-02} we have $\psi\big(\zeta_n(x,b_2^{-1}\big)\big)=0$, while $\tilde{E}_{0,n}\big(\zeta_n\big(x,b_2^{-1}\big)\big)\ne 0$. Therefore $c'_{0n}=0$.

Next suppose that $c'_{lm}=0$ for $(l,m)\in D_{j+1}$.
Then using Lemma~\ref{lem:tri-zeta(xb)} for (\ref{eq:psi(z)}) at $z=\zeta_j\big(x,b_2^{-1}\big)$
we have
\begin{equation*}
\psi\big(\zeta_j\big(x,b_2^{-1}\big)\big)=\sum_{l=0}^{n-j}c'_{lj}\tilde{E}_{l,j}\big(\zeta_j\big(x,b_2^{-1}\big)\big)=
\left(\sum_{l=0}^{n-j}c'_{lj}x^lt^{l(n-j)-{l+1\choose 2}}\right)\tilde{E}_{0,j}\big(\zeta_j\big(x,b_2^{-1}\big)\big).
\end{equation*}
From Lemma \ref{lem:3term1st-02} $\psi\big(\zeta_j\big(x,b_2^{-1}\big)\big)$ vanishes as a function of $x$,
while $\tilde{E}_{0,j}\big(\zeta_j\big(x,b_2^{-1}\big)\big)\ne 0$.
Thus, $\sum_{l=0}^{n-j}c'_{lj}x^lt^{l(n-j)-{l+1\choose 2}}=0$, i.e., the coefficient $c'_{lj}t^{l(n-j)-{l+1\choose 2}}$ of $x^l$ vanishes for $0\le l\le n-j$. Therefore $c'_{lj}=0$ for $0\le l\le n-j$.
This implies
$c'_{lm}=0$ for $(l,m)\in D_{j}$.

On the other hand, we prove \eqref{eq:expand04} of Lemma \ref{lem:3term1st-01}.
Set $D'_j=\big\{(l,i)\in \mathbb{Z}^2\,|\, n\le i+l,0\le l\le n$, $0\le i\le j\big\}$,
which satisfies $\{(n,0)\}=D'_0\subset D'_1\subset\cdots\subset D'_n$.
The set $\{\tilde{E}_{k,i}(z)\,|\, (k,i)\in D'_n\}$ also forms a basis for the linear space spanned
by $\{m_\lambda(z)\,|\, \lambda\le (2^n)\}$.
If we put
\begin{equation*}
\psi'(z):=
\frac{\tilde{\varphi}_{k,i}(z)}{\Delta(z)}
-\big(d_{k,i}\tilde{E}_{k,i}(z)+d_{k,i+1}\tilde{E}_{k,i+1}(z)+d_{k-1,i+1}\tilde{E}_{k-1,i+1}(z)\big),
\end{equation*}
where $d_{k,i}, d_{k,i+1}$ and $d_{k-1,i+1}$ are specified by (\ref{eq:dk,i})--(\ref{eq:dk-1,i+1}), then the symmetric polynomial $\psi'(z)$ is expressed as a linear combination of $\tilde{E}_{k,i}(z)$, $(k,i)\in D'_n$, i.e.,
\begin{equation}\label{eq:psi'(z)}
\psi'(z)=\sum_{(l,m)\in D'_n}d'_{lm}\tilde{E}_{l,m}(z),
\end{equation}
where the coefficients $d'_{lm}$ are some constants. We now prove $\psi'(z)=0$ identically, i.e.,
$d'_{lm}=0$ for all $(l,m)\in D'_n$ inductively.
Namely, we prove that, if $d'_{lm}=0$ for $(l,m)\in D'_{j-1}$, then $d'_{lm}=0$ for $(l,m)\in D'_{j}.$

First we show that $d'_{n0}=0$ as the starting point of induction.
Using Lemma \ref{lem:tri-zeta(ay)} for~(\ref{eq:psi'(z)}) at $z=\zeta_0((a_1,y))$ we have
$\psi'(\zeta_0(a_1,y))=d'_{n0}\tilde{E}_{n,0}(\zeta_0(a_1,y))$.
From Lemma~\ref{lem:3term1st-02} we have $\psi'(\zeta_0(a_1,y))\allowbreak =0$, while $\tilde{E}_{n,0}(\zeta_0(a_1,y))\ne 0$. Therefore $d'_{n0}=0$.

Next suppose that $d'_{lm}=0$ for $(l,m)\in D'_{j-1}$.
Then using Lemma~\ref{lem:tri-zeta(ay)} for~(\ref{eq:psi'(z)}) at $z=\zeta_j(a_1,y)$
we have
\begin{gather*}
\psi'(\zeta_j(a_1,y))=\sum_{l=n-j}^nd'_{lj}\tilde{E}_{l,j}(\zeta_j(a_1,y))\\
\hphantom{\psi'(\zeta_j(a_1,y))}{} =
\left(\sum_{l=n-j}^nd'_{lj}y^{l+j-n}a_1^{n-j}t^{{n-j\choose 2}-{l+j-n\choose 2}}\right)\tilde{E}_{0,j}(\zeta_j(a_1,y)).
\end{gather*}
From Lemma \ref{lem:3term1st-02} $\psi'(\zeta_j(a_1,y))$ vanishes as a function of~$y$,
while $\tilde{E}_{0,j}(\zeta_j(a_1,y))\ne 0$. Thus,
$\sum_{l=n-j}^nd'_{lj}y^{l+j-n}a_1^{n-j}t^{{n-j\choose 2}-{l+j-n\choose 2}}=0$,
i.e., the coefficient $d'_{lj}a_1^{n-j}t^{{n-j\choose 2}-{l+j-n\choose 2}}$ of $y^{l+j-n}$ va\-nishes for $n-j\le l\le n$.
Therefore $d'_{lj}=0$ for $n-j\le l\le n$.
This implies
$d'_{lm}=0$ for $(l,m)\in D'_{j}$.
\end{proof}

\section[The transition matrix $R$]{The transition matrix $\boldsymbol{R}$}\label{section06}

In this section we give a proof of Theorem \ref{thm:R=LDU}.
Before proving Theorem \ref{thm:R=LDU}, we will show the results deduced from Theorem \ref{thm:R=LDU}.
By the definition \eqref{eq:R} of the transition matrix $R$, we have
\begin{equation*}
R^{-1}=J\bar{R}J,
\end{equation*}
where the symbol $\bar{R}$ is the matrix $R$ after the interchange $(a_1,b_1)\leftrightarrow (a_2,b_2)$ and
$J$ is the matrix specified by
\begin{equation*}
J=\begin{pmatrix}
 & & &1 \\
 & &1 & \\
 & \cdots & & \\
1 & & &
\end{pmatrix}.
\end{equation*}
The explicit form of the inverse matrix of $R$ is given by
\begin{Corollary}
\label{cor:inverseR}
The inverse matrix $R^{-1}$ is written as Gauss matrix decomposition
\begin{equation*}
R^{-1}=U_\tinyR^{-1}D_\tinyR^{-1}L_\tinyR^{-1}=L'{}_{\!\!\tinyR}^{-1}D'{}_{\!\!\tinyR}^{-1}U'{}_{\!\!\!\tinyR}^{-1},
\end{equation*}
where the inverse matrices
$L_\tinyR^{-1}=\big(l^{\tinyR *}_{ij}\big)_{0\le i,j\le n}$,
$D_\tinyR^{-1}=\big(d^{\tinyR *}_{j}\delta_{ij}\big)_{0\le i,j\le n}$,
$U_\tinyR^{-1}=\big(u^{\tinyR *}_{ij}\big)_{0\le i,j\le n}$ are lower triangular, diagonal, upper triangular, respectively, given by
\begin{subequations}
\begin{gather}
l^{\tinyR *}_{ij}=\overline{u^\tinyR_{n-i,n-j}}=
\qbin{n-j}{n-i}{t^{-1}}
\frac{\big(a_2b_2t^{j};t\big)_{i-j}}{\big(a_2a_1^{-1}t^{i+j-n};t\big)_{i-j}},\label{eq:l*Rij}\\
d^{\tinyR *}_{j}=\overline{d^\tinyR_{n-j}}=
\frac{\big(a_2a_1^{-1}t^{-(n-j)};t\big)_{j}(a_1b_2;t)_{n-j}}{(a_2b_1;t)_{j}\big(a_2^{-1}a_1t^{-j};t\big)_{n-j}},
\label{eq:d*Rij}\\
u^{\tinyR *}_{ij}=\overline{l^\tinyR_{n-i,n-j}}=
(-1)^{j-i}t^{-{j-i\choose 2}}
\qbin{j}{i}{t^{-1}}
\frac{\big(a_1b_1t^{n-j};t\big)_{j-i}}{\big(a_2^{-1}a_1 t^{n-2j+1};t\big)_{j-i}},\label{eq:u*Rij}
\end{gather}
\end{subequations}
and the inverse matrices
$U'{}_{\!\!\!\tinyR}^{-1}=\big(u^{\tinyR\,\prime*}_{ij}\big)_{0\le i,j\le n}$,
$D'{}_{\!\!\tinyR}^{-1}=\big(d^{\tinyR\,\prime*}_{j}\delta_{ij}\big)_{0\le i,j\le n}$,
$L'{}_{\!\!\tinyR}^{-1}=\big(l^{\tinyR\,\prime*}_{ij}\big)_{0\le i,j\le n}$
are upper triangular, diagonal, lower triangular, respectively, given by
\begin{subequations}
\begin{gather}
u^{\tinyR\,\prime*}_{ij}=\overline{l^{\tinyR\,\prime}_{n-i,n-j}}=
\qbin{j}{i}{t}
\frac{\big(a_1^{-1}b_1^{-1}t^{-(n-i-1)};t\big)_{j-i}}{\big(b_2b_1^{-1}t^{-(n-2i-1)};t\big)_{j-i}},\label{eq:u'*Rij}\\
d^{\tinyR\,\prime*}_{j}=\overline{d^{\tinyR\,\prime}_{n-j}}=\frac{\big(b_2b_1^{-1}t^{-(n-2j-1)};t\big)_{n-j} \big(a_1^{-1}b_2^{-1}t^{-(j-1)};t\big)_j}
{\big(a_2^{-1}b_1^{-1}t^{-(n-j-1)};t\big)_{n-j}\big(b_2^{-1}b_1t^{n-2j+1};t\big)_j},\label{eq:d'*Rij}\\
l^{\tinyR\,\prime*}_{ij}=\overline{u^{\tinyR\,\prime}_{n-i,n-j}}
=(-1)^{i-j}t^{i-j\choose 2}\qbin{n-j}{n-i}{t} \frac{\big(a_2^{-1}b_2^{-1}t^{-(i-1)};t\big)_{i-j}}{\big(b_2^{-1}b_1t^{n-i-j};t\big)_{i-j}}.\label{eq:l'*Rij}
\end{gather}
\end{subequations}
\end{Corollary}

\begin{proof}
Since $R^{-1}=J\bar{R}J$, we have $R^{-1}=U_\tinyR^{-1}D_\tinyR^{-1}L_\tinyR^{-1}$,
where
$L_\tinyR^{-1}=J\overline{U_\tinyR} J$,
$D_\tinyR^{-1}=J\overline{D_\tinyR}J$ and
$U_\tinyR^{-1}=J\overline{L_\tinyR}J$.
Thus we immediately have the expressions
$l^{\tinyR *}_{ij}=\overline{u^\tinyR_{n-i,n-j}}$,
$d^{\tinyR *}_{j}=\overline{d^\tinyR_{n-j}}$ and
$u^{\tinyR *}_{ij}=\overline{l^\tinyR_{n-i,n-j}}$.
From Theorem \ref{thm:R=LDU} this gives the explicit forms \eqref{eq:l*Rij}, \eqref{eq:d*Rij} and \eqref{eq:u*Rij}. On the other hand, we also have
$R^{-1}=L'{}_{\!\!\tinyR}^{-1}D'{}_{\!\!\tinyR}^{-1}U'{}_{\!\!\!\tinyR}^{-1}$,
where
$U'{}_{\!\!\!\tinyR}^{-1}=J\overline{L'{}_{\!\!\tinyR}}J$,
$D'{}_{\!\!\tinyR}^{-1}=J\overline{D'{}_{\!\!\tinyR}}J$
and
$L'{}_{\!\!\tinyR}^{-1}=J\overline{U'{}_{\!\!\!\tinyR}}J$.
Therefore
$u^{\tinyR\,\prime*}_{ij}=\overline{l^{\tinyR\,\prime}_{n-i,n-j}}$,
$d^{\tinyR\,\prime*}_{j}=\overline{d^{\tinyR\,\prime}_{n-j}}$,
and
$l^{\tinyR\,\prime*}_{ij}=\overline{u^{\tinyR\,\prime}_{n-i,n-j}}$.
Thus we obtain the expressions \eqref{eq:u'*Rij}, \eqref{eq:d'*Rij} and \eqref{eq:l'*Rij}.
\end{proof}

The rest of this section is devoted to the proof of Theorem \ref{thm:R=LDU}.
For this purpose we introduce another set of symmetric polynomials different from Matsuo's polynomials.

For $0\le r\le n$, let $f_r(a_1,a_2;t;z)$
be (symmetric) polynomials specified by
\begin{equation}\label{eq:fr(a1,a2;t;z)01}
f_r(a_1,a_2;t;z)
:=
\sum_{\substack{I\subseteq\{1,\ldots,n\}\\[1pt] |I|=r}}\,
\prod_{k=1}^r\frac{z_{i_k}-a_2t^{i_k-k}}{a_1t^{k-1}-a_2t^{i_k-k}}
\prod_{l=1}^{n-r}\frac{z_{j_l}-a_1t^{j_l-l}}{a_2t^{l-1}-a_1t^{j_l-l}},
\end{equation}
where the summation is over all $r$-subsets $I$ of $\{1,\ldots,n\}$, and
$I=\{i_1<\cdots<i_r\}$, $J=\{1,\ldots,n\}\backslash I=\{ j_1<\cdots<j_{n-r}\}$.
In particular,
\begin{equation*}
f_0(a_1,a_2;t;z)=\prod_{i=1}^n\frac{z_i-a_1}{a_2t^{i-1}-a_1}, \qquad
f_n(a_1,a_2;t;z)=\prod_{i=1}^n\frac{z_i-a_2}{a_1t^{i-1}-a_2}.
\end{equation*}
We remark that the polynomials $f_r(a_1,a_2;t;z)$ are called the {\it Lagrange interpolation polynomials of type $A$} and
their properties are discussed
in \cite[Appendix B]{IN2018}.
By definition the polynomial $f_i(a_1,a_2;t;z)$ satisfies
\begin{equation}
\label{eq:f_i=f_n-i}
f_i(a_1,a_2;t;z)=f_{n-i}(a_2,a_1;t;z).
\end{equation}
When we need to specify the number of variables $z_1,\ldots,z_n$,
we use the notation $f_i^{(n)}(a_1,a_2;t;z)\allowbreak =f_i(a_1,a_2;t;z)$.
\begin{Lemma}[recurrence relation]\label{lem:rec}The polynomials \eqref{eq:fr(a1,a2;t;z)01} satisfy the following recurrence relations:
\begin{gather*}
f_i^{(n)}(a_1,a_2;t;z)
=\frac{z_n-a_2t^{n-i}}{a_1t^{i-1}-a_2t^{n-i}}
f_{i-1}^{(n-1)}(a_1,a_2;t;
\widehat{z}_n)
+\frac{z_n-a_1t^i}{a_2t^{n-i-1}-a_1t^i}
f_{i}^{(n-1)}(a_1,a_2;t;
\widehat{z}_n)
\end{gather*}
for $i=0,1,\ldots,n$, where
$
\widehat{z}_n
=(z_1,\ldots,z_{n-1})\in (\mathbb{C}^*)^{n-1}$.
\end{Lemma}

\begin{proof}The lemma follows from a direct computation and we omit the detail.
\end{proof}

For arbitrary $x,y\in \mathbb{C}^*$ we define
\begin{equation}\label{eq:xi_j(x,y;t)}
\xi_j(x,y;t):=\big(\underbrace{x,x t,\ldots,x t^{j-1}\phantom{\Big|}\!\!}_j,
\underbrace{y,y t,\ldots,y t^{n-j-1}\phantom{\Big|}\!\!}_{n-j}\big)\in (\mathbb{C}^*)^n
\end{equation}
for $j=0,1,\ldots, n$.

\begin{Proposition}\label{prop:vanishing-f}
The polynomial $f_i(a_1,a_2;t;z)$ is symmetric in the variables $z=(z_1,\ldots,z_n)$.
The leading term of $f_i(a_1,a_2;t;z)$ is $m_{(1^n)}(z)$ up to a multiplicative constant.
The functions $f_i(a_1,a_2;t;z)$ $(i=0,1,\ldots,n)$ satisfy
\begin{equation}
\label{eq:vanishing-f}
f_i(a_1,a_2;t;\xi_j(a_1,a_2;t))=\delta_{ij}.
\end{equation}
\end{Proposition}

\begin{proof}See \cite[Example 4.3 and equation~(4.7)]{IN2018}. Otherwise,
using Lemma~\ref{lem:rec} we can also prove this proposition directly by induction on~$n$.
\end{proof}
\begin{Remark}The set of symmetric polynomials $\{f_i(a_1,a_2;t;z)\,|\,i=0,1,\ldots,n\}$
forms a basis of the linear space spanned by $\{m_\lambda(z)\,|\,\lambda\le (1^n)\}$.
Conversely such basis satisfying the condition~\eqref{eq:vanishing-f} is uniquely determined.
Thus we can take Proposition~\ref{prop:vanishing-f} as a definition of the polynomials $f_i(a_1,a_2;t;z)$,
instead of~(\ref{eq:fr(a1,a2;t;z)01}).
\end{Remark}

\begin{Lemma}[triangularity]\label{lem:Triangularity-f}
Suppose that
\begin{equation*}
\xi_j(a_1):=\big(\underbrace{a_1,a_1 t,\ldots,a_1 t^{j-1}\phantom{\Big|}\!\!}_{j},z_1,z_2,\ldots,z_{n-j}\big)
\in (\mathbb{C}^*)^n.
\end{equation*}
If $i<j$, then
\begin{equation}\label{eq:f_i=0xi}
f_i(a_1,a_2;t;\xi_j(a_1))=0.
\end{equation}
Moreover, $f_i(a_1,a_2;t;\xi_i(a_1))$ evaluates as
\begin{equation}\label{eq:fi(xi_i(a_1))}
f_i(a_1,a_2;t;\xi_i(a_1))=\prod_{l=1}^{n-i}\frac{z_l-a_1t^{i}}{a_2t^{l-1}-a_1t^{i}}
=\frac{\prod_{l=1}^{n-i}\big(1-z_la_1^{-1}t^{-i}\big)}{\big(a_2a_1^{-1}t^{-i};t\big)_{n-i}}.
\end{equation}
On the other hand, suppose that
\begin{equation*}
\eta_j(a_2):=\big(z_1,z_2,\ldots,z_j, \underbrace{a_2,a_2 t,\ldots,a_2 t^{n-j-1}\phantom{\Big|}\!\!}_{n-j}\big)
\in (\mathbb{C}^*)^n.
\end{equation*}
If $i>j$, then
\begin{equation}\label{eq:f_i=0eta}
f_i(a_1,a_2;t;\eta_j(a_2))=0.
\end{equation}
Moreover, $f_i(a_1,a_2;t;\eta_i(a_2))$ evaluates as
\begin{equation}\label{eq:fi(eta_i(a_2))}
f_i(a_1,a_2;t;\eta_i(a_2))=\prod_{l=1}^i\frac{z_l-a_2t^{n-i}}{a_1t^{l-1}-a_2t^{n-i}}
=\frac{\prod_{l=1}^i\big(1-z_la_2^{-1}t^{-(n-i)}\big)}{\big(a_1a_2^{-1}t^{-(n-i)};t\big)_i}.
\end{equation}
\end{Lemma}

\begin{proof}First we show \eqref{eq:f_i=0eta} by induction on $n$.
For simplicity we write $\eta_i(a_2)$ as $\eta_i$.
Suppose $i>j$. Using Lemma~\ref{lem:rec} we have
\begin{gather*}
 f_i^{(n)}(a_1,a_2;t;\eta_j)
 =\frac{a_2t^{n-j-1}-a_2t^{n-i}}{a_1t^{i-1}-a_2t^{n-i}}
f_{i-1}^{(n-1)}\big(a_1,a_2;t;\eta_j^{(n-1)}\big)\\
\hphantom{f_i^{(n)}(a_1,a_2;t;\eta_j) =}{}
+\frac{a_2t^{n-j-1}-a_1t^i}{a_2t^{n-i-1}-a_1t^i}
f_{i}^{(n-1)}\big(a_1,a_2;t;\eta_j^{(n-1)}\big),
\end{gather*}
where $\eta_j^{(n-1)}=\big(z_1,z_2,\ldots,z_j,
a_2,a_2 t,\ldots,a_2 t^{n-j-2}
\big)\in (\mathbb{C}^*)^{n-1}$.
Since $f_{i}^{(n-1)}\big(a_1,a_2;t;\eta_i^{(n-1)}\big)=0$
by the induction hypothesis,
we have
\begin{equation*}
f_i^{(n)}(a_1,a_2;t;\eta_j)
=\frac{a_2t^{n-j-1}-a_2t^{n-i}}{a_1t^{i-1}-a_2t^{n-i}}
f_{i-1}^{(n-1)}\big(a_1,a_2;t;\eta_j^{(n-1)}\big).
\end{equation*}
If $i-1>j$, then $f_{i-1}^{(n-1)}\big(a_1,a_2;t;\eta_j^{(n-1)}\big)=0$
by the induction hypothesis,
while if $i-1=j$, then $a_2t^{n-j-1}-a_2t^{n-i}=0$.
In any case we obtain $f_i^{(n)}(a_1,a_2;t;\eta_i)=0$, which is the claim of~\eqref{eq:f_i=0eta}.

Next we show \eqref{eq:fi(eta_i(a_2))}.
If we put $z_l=a_2 t^{n-i}$ for $l\in \{1,\ldots,i\}$ in the polynomial $f_i(a_1,a_2;t;\allowbreak\eta_i(a_2))$ of $z_1,\ldots, z_i$,
then we have $f_i(a_1,a_2;t; \eta_i(a_2))=0$ because $f_i(a_1,a_2;t;\eta_i(a_2))|_{z_k=a_2t^{n-i}}$ satisfies
the condition of~\eqref{eq:f_i=0eta}.
This implies
$f_i(a_1,a_2;t;\eta_i(a_2))$ is divisible by
$\prod_{l=1}\big(z_l-a_2t^{n-i}\big)$, so that we have
$f_i(a_1,a_2;t;\eta_i(a_2))=c\prod_{l=1}^i\big(z_l-a_2t^{n-i}\big)$,
where $c$ is some constant.
Thus
\begin{equation*}
f_i(a_1,a_2;t;\eta_i(a_2))\Big|_{(z_1,\ldots,z_i)=(a_1,a_1t,\ldots,a_1t^{i-1})}
=c\prod_{l=1}^i\big(a_1t^{l-1}-a_2t^{n-i}\big).
\end{equation*}
On the other hand, \eqref{eq:vanishing-f} implies that
\begin{equation*}
f_i(a_1,a_2;t;\eta_i(a_2))\Big|_{(z_1,\ldots,z_i)=(a_1,a_1t,\ldots,a_1t^{i-1})}
=f_i(a_1,a_2;t;\xi_i(a_1,a_2;t))=1.
\end{equation*}
We therefore obtain $c=1/\prod_{l=1}^i\big(a_1t^{l-1}-a_2t^{n-i}\big)$, which implies~\eqref{eq:fi(eta_i(a_2))}.

Finally we show \eqref{eq:f_i=0xi} and \eqref{eq:fi(xi_i(a_1))}.
From \eqref{eq:f_i=f_n-i} we have
\begin{equation*}
f_i(a_1,a_2;t;\xi_j(a_1))=
f_{n-i}(a_2,a_1;t;\xi_j(a_1))=
f_{n-i}(a_2,a_1;t;\eta_{n-j}(a_1)).
\end{equation*}
If $i<j$ (i.e., $n-i>n-j$), then using \eqref{eq:f_i=0eta} we see that the right-hand side of the above is equal to zero.
Moreover, using \eqref{eq:fi(eta_i(a_2))} we obtain
\begin{equation*}
f_i(a_1,a_2;t;\xi_i(a_1))=f_{n-i}(a_2,a_1;t;\eta_{n-i}(a_1))
=\prod_{l=1}^{n-i}\frac{z_l-a_1t^{i}}{a_2t^{l-1}-a_1t^{i}}
=\frac{\prod_{l=1}^{n-i}\big(1-z_la_1^{-1}t^{-i}\big)}{\big(a_2a_1^{-1}t^{-i};t\big)_{n-i}},
\end{equation*}
which completes the proof.
\end{proof}

\begin{Corollary}\label{cor:Triangularity-f2}
Let $\xi_j(x,a_2;t)\in (\mathbb{C}^*)^n$ be the point specified by~\eqref{eq:xi_j(x,y;t)} with $y=a_2$. Then
$f_i(a_1,a_2;t;\xi_j(x,a_2;t))$ evaluates as
\begin{equation}\label{eq:f_i(xi_j(x,a_2;t))}
f_i(a_1,a_2;t;\xi_j(x,a_2;t))
=
\qbin{j}{i}{t^{-1}}
\frac{\big(xa_1^{-1};t\big)_{j-i}\big(xa_2^{-1}t^{-(n-j)};t\big)_i}
{\big(a_1^{-1}a_2t^{n-j-i};t\big)_{j-i}\big(a_1a_2^{-1}t^{-(n-i)};t\big)_i}.
\end{equation}
\end{Corollary}

\begin{proof}If $i>j$, then
$f_i(a_1,a_2;t;\xi_j(x,a_2;t))=0$
is a special case of \eqref{eq:f_i=0eta} in Lemma~\ref{lem:Triangularity-f}.
Now we assume that $i\le j$.
If we put $x=a_2t^{n-j-k}$ $(k=0,1,\ldots, i-1)$, then from~\eqref{eq:f_i=0eta} we have $f_i(a_1,a_2;t;\xi_j(x,a_2;t))=0$.
If we put $x=a_1t^{-k}$ $(k=0,1,\ldots, j-i-1)$, then from~\eqref{eq:f_i=0xi} we also have $f_i(a_1,a_2;t;\xi_j(x,a_2;t))=0$.
This implies that $f_i(a_1,a_2;t;\xi_j(x,a_2;t))$ as a polynomial of $x$ is divisible by $(xa_1)_{j-i}\big(xa_2^{-1}t^{-(n-j)}\big)_i$.
Since the degree of $f_i(a_1,a_2;t;\xi_j(x,a_2;t))$ as a function of $x$ is equal to $j$, the function $f_i(a_1,a_2;t;\xi_j(x,a_2;t))$
can be expressed as
\begin{equation*}
f_i(a_1,a_2;t;\xi_j(x,a_2;t))=c\big(xa_1^{-1};t\big)_{j-i}\big(xa_2^{-1}t^{-(n-j)};t\big)_i,
\end{equation*}
where $c$ is some constant. In order to determine the constant $c$, we put $x=a_2t^{n-j-i}$
in the above equation. Then
\begin{equation*}
f_i(a_1,a_2;t;\xi_j(x,a_2;t))\Big|_{x=a_2t^{n-j-i}}=c\big(a_1^{-1}a_2t^{n-j-i};t\big)_{j-i}\big(t^{-i};t\big)_i,
\end{equation*}
while, from \eqref{eq:fi(eta_i(a_2))} we have
\begin{gather*}
f_i(a_1,a_2;t;\xi_j(x,a_2;t))
\Big|_{x=a_2t^{n-j-i}}
=
f_i(a_1,a_2;t;\xi_i(x,a_2;t))
\Big|_{x=a_2t^{n-j-i}}
=\frac{\big(t^{-j};t\big)_i}{\big(a_1a_2^{-1}t^{-(n-i)};t\big)_i}.
\end{gather*}
The constant $c$ can be explicitly computed as
\begin{align*}
c&=\frac{\big(t^{-j};t\big)_i}{\big(a_1^{-1}a_2t^{n-j-i};t\big)_{j-i}\big(a_1a_2^{-1}t^{-(n-i)};t\big)_i\big(t^{-i};t\big)_i}\\
&=\frac{\big(t^{-1};t^{-1}\big)_j}{\big(a_1^{-1}a_2t^{n-j-i};t\big)_{j-i}\big(a_1a_2^{-1}t^{-(n-i)};t\big)_i
\big(t^{-1};t^{-1}\big)_{j-i}\big(t^{-1};t^{-1}\big)_i}.
\end{align*}
We therefore obtain \eqref{eq:f_i(xi_j(x,a_2;t))}.
\end{proof}

\begin{Lemma}\label{lem:(e)=(f)tildeU}
Suppose that $\tilde{U}_\tinyR$ is the $(n+1)\times(n+1)$ matrix satisfying
\begin{gather}
 \big(e_n(a_2,b_1;z),e_{n-1}(a_2,b_1;z),\ldots,e_0(a_2,b_1;z)\big)\nonumber\\
 \qquad{}=\big(f_n(a_1,a_2;t;z),f_{n-1}(a_1,a_2;t;z),\ldots,f_0(a_1,a_2;t;z)\big)\tilde{U}_\tinyR.\label{eq:(e)=(f)tildeU}
\end{gather}
Then $\tilde{U}_\tinyR=\big(\tilde{u}^\tinyR_{ij}\big)_{0\le i,j\le n}$ is an upper triangular matrix
with entries given by
\begin{equation}\label{eq:tilde u}
\tilde{u}^\tinyR_{ij}
=
\frac{\big(a_1b_1t^{n-j};t\big)_{j-i}\big(a_1a_2^{-1}t^{-i};t\big)_{n-j}(a_2b_1;t)_i\big(t^{-1};t^{-1}\big)_{n}}
{\big(1-t^{-1}\big)^n}\frac{\qbin{j}{i}{t^{-1}}}{\qbin{n}{i}{t^{-1}}}.
\end{equation}
Suppose that $\tilde{L}_\tinyR$ is the $(n+1)\times(n+1)$ matrix satisfying
\begin{gather}
\big(f_n(a_1,a_2;t;z),f_{n-1}(a_1,a_2;t;z),\ldots,f_0(a_1,a_2;t;z)\big)\nonumber\\
 \qquad=\big(e_0(a_1,b_2;z),e_1(a_1,b_2;z),\ldots,e_n(a_1,b_2;z)\big)\tilde{L}_\tinyR.\label{eq:(f)=(e)tildeL}
\end{gather}
Then $\tilde{L}_\tinyR=\big(\tilde{l}^\tinyR_{ij}\big)_{0\le i,j\le n}$ is a lower triangular matrix
with entries given by
\begin{gather}\label{eq:tilde l}
\tilde{l}^\tinyR_{ij}
=
\frac{(-1)^{i-j}t^{-{i-j\choose 2}}\big(a_2b_2t^j;t\big)_{i-j}\big(1-t^{-1}\big)^n}
{\big(a_1^{-1}a_2t^{-(n-2j-1)};t\big)_{i-j}(a_1b_2;t)_{n-j}\big(a_1^{-1}a_2t^{-(n-j)};t\big)_j\big(t^{-1};t^{-1}\big)_{n}}
\qbin{n}{i}{t^{-1}}\qbin{i}{j}{t^{-1}}.\!\!\!
\end{gather}
\end{Lemma}

\begin{proof}Since both $\{e_i(a_2,b_1;z)\,|\,i=0,1,\ldots,n\}$ and
$\{f_i(a_1,a_2;t;z)\,|\,i=0,1,\ldots,n\}$ form bases of the linear space spanned by $\{m_\lambda(z)\,|\,\lambda<(1^n)\}$,
the polynomial $e_i(a_2,b_1;z)$ is expressed as a~linear combination of
$f_i(a_1,a_2;t;z)$ $(i=0,1,\ldots,n)$, i.e.,
\begin{equation*}
e_{n-j}(a_2,b_1;z)=\sum_{i=0}^{n}
\tilde{u}^\tinyR_{ij}
f_{n-i}(a_1,a_2;t;z),
\end{equation*}
where $\tilde{u}^\tinyR_{ij}$ are some constants.
From the vanishing property~\eqref{eq:vanishing-f}, the coefficient~$\tilde{u}^\tinyR_{ij}$
is given by
\begin{equation}\label{eq:tilde u-2}
\tilde{u}^\tinyR_{ij}
=e_{n-j}(a_2,b_1;\xi_{n-i}(a_1,a_2;t))
=e_{n-j}(a_2,b_1;\zeta_{n-i}(a_2,y))\Big|_{y=a_1t^{n-i-1}}.
\end{equation}
From \eqref{eq:tri-zeta(ay)2} in Lemma \ref{lem:tri-zeta(ay)}, $e_{n-j}(a_2,b_1;\zeta_{n-i}(a_2,y))$
evaluates as
\begin{gather*}
e_{n-j}(a_2,b_1;\zeta_{n-i}(a_2,y))\\
\qquad{} =
\big(yb_1t^{-(j-i-1)};t\big)_{j-i}\big(ya_2^{-1}t^{-(n-1)};t\big)_{n-j}(a_2b_1;t)_i
\frac{\big(t^{-1};t^{-1}\big)_{n-i}\big(t^{-1};t^{-1}\big)_{j}}{\big(t^{-1};t^{-1}\big)_{j-i}\big(1-t^{-1}\big)^n}.
\end{gather*}
Combining this and \eqref{eq:tilde u-2}, we obtain the expression~\eqref{eq:tilde u}.

Since $\{e_i(a_1,b_2;z)\,|\,i=0,1,\ldots,n\}$ is also a basis of the linear space spanned by $\{m_\lambda(z)\,|\,\lambda<(1^n)\}$,
the polynomial $f_i(a_1,a_2;t;z)$ is expressed as a linear combination of
$e_i(a_1,b_2;z)$ $(i=0,1,\ldots,n)$, i.e.,
\begin{equation*}
f_{n-j}(a_1,a_2;t;z)=\sum_{i=0}^{n}\tilde{l}^\tinyR_{ij}e_{i}(a_1,b_2;z),
\end{equation*}
where $\tilde{l}^\tinyR_{ij}$ are some constants.
From \eqref{eq:matsuo2}, the coefficient $\tilde{l}^\tinyR_{ij}$ is written as
\begin{equation}\label{eq:tilde l-2}
\tilde{l}^\tinyR_{ij}
=\frac{f_{n-j}\big(a_1,a_2;t;\zeta_i\big(a_1,b_2^{-1}\big)\big)}{c_i}=
c_i^{-1}f_{n-j}(a_1,a_2;t;\xi_{n-i}(a_1,x))\Big|_{x=b_2^{-1}t^{-(i-1)}},
\end{equation}
where $c_i$ is given explicitly in \eqref{eq:c-i-b} as
\begin{equation}\label{eq:c-i-2}
c_i=
(a_1b_2;t)_{n-i}\big(a_1^{-1}b_2^{-1}t^{-(n-1)};t\big)_{i}
\frac{\big(t^{-1};t^{-1}\big)_i\big(t^{-1};t^{-1}\big)_{n-i}} {\big(1-t^{-1}\big)^n}.
\end{equation}
Using \eqref{eq:f_i(xi_j(x,a_2;t))} in Corollary~\ref{cor:Triangularity-f2}, we have
\begin{align*}
f_{n-j}(a_1,a_2;t;\xi_{n-i}(a_1,x))&=
f_{j}(a_2,a_1;t;\xi_{n-i}(a_1,x))=f_{j}(a_2,a_1;t;\xi_{i}(x,a_1))\\
&=
\qbin{i}{j}{t^{-1}}
\frac{\big(xa_2^{-1};t\big)_{i-j}\big(xa_1^{-1}t^{-(n-i)};t\big)_j}
{\big(a_2^{-1}a_1t^{n-i-j};t\big)_{i-j}\big(a_2a_1^{-1}t^{-(n-j)};t\big)_j}.
\end{align*}
Combining this, \eqref{eq:tilde l-2} and \eqref{eq:c-i-2}, we therefore obtain the expression
\eqref{eq:tilde l}.
\end{proof}

\begin{Lemma}\label{lem:(e)=(f)tildeL'}
Suppose that $\tilde{L}'{}_{\!\!\tinyR}$ is the $(n+1)\times(n+1)$ matrix satisfying
\begin{gather}
\big(e_n(a_2,b_1;z),e_{n-1}(a_2,b_1;z),\ldots,e_0(a_2,b_1;z)\big)\nonumber\\
 \qquad{} =\big(f_n\big(b_1^{-1},b_2^{-1};t^{-1};z\big),f_{n-1}\big(b_1^{-1},b_2^{-1};t^{-1};z\big),\ldots, f_0\big(b_1^{-1},b_2^{-1};t^{-1};z\big)\big)
\tilde{L}'{}_{\!\!\tinyR}.\label{eq:(e)=(f)tildeL'}
\end{gather}
Then $\tilde{L}'{}_{\!\!\tinyR}=\big(\tilde{l}^{\tinyR\,\prime}_{ij}\big)_{0\le i,j\le n}$ is a lower triangular matrix
with entries given by
\begin{equation}\label{eq:c-(e)=(f)tildeL'}
\tilde{l}^{\tinyR\,\prime}_{ij}
=
\frac{\big(a_2^{-1}b_2^{-1}t^{-(i-1)};t\big)_{i-j}\big(a_2^{-1}b_1^{-1}t^{-(n-i-1)};t\big)_{n-i}\big(b_1b_2^{-1}t^{n-i-j+1};t\big)_{j}
(t;t)_n}{t^{{n\choose 2}}(1-t)^n}
\frac{\qbin{i}{j}{t}}{\qbin{n}{j}{t}}.
\end{equation}
Suppose that $\tilde{U}'{}_{\!\!\!\tinyR}$ is the $(n+1)\times(n+1)$ matrix satisfying
\begin{gather}
\big(f_n\big(b_1^{-1},b_2^{-1};t^{-1};z\big),f_{n-1}\big(b_1^{-1},b_2^{-1};t^{-1};z\big),\ldots,f_0\big(b_1^{-1},b_2^{-1};t^{-1};z\big)\big)
\nonumber\\
\qquad=\big(e_0(a_1,b_2;z),e_1(a_1,b_2;z),\ldots,e_n(a_1,b_2;z)\big)
\tilde{U}'{}_{\!\!\!\tinyR},\label{eq:(f)=(e)tildeU'}
\end{gather}
then $\tilde{U}'{}_{\!\!\!\tinyR}=\big(\tilde{u}^{\tinyR\,\prime}_{ij}\big)_{0\le i,j\le n}$ is an upper triangular matrix
with entries given by
\begin{gather}
\tilde{u}^{\tinyR\,\prime}_{ij}=
\frac{(-1)^{j-i}t^{{j-i\choose 2}+{n\choose 2}}\big(a_1^{-1}b_1^{-1}t^{-(n-i-1)};t\big)_{j-i}(1-t)^n}
{\big(b_1^{-1}b_2t^{-(n-i-j)};t\big)_{j-i}\big(b_1^{-1}b_2t^{-(n-2j-1)};t\big)_{n-j}\big(a_1^{-1}b_2^{-1}t^{-(j-1)};t\big)_j(t;t)_n}\nonumber\\
\hphantom{\tilde{u}^{\tinyR\,\prime}_{ij}=}{}\times
\qbin{n}{j}{t}\qbin{j}{i}{t}.\label{eq:c-(f)=(e)tildeU'}
\end{gather}
\end{Lemma}

\begin{proof}Since both $\{e_i(a_2,b_1;z)\,|\,i=0,1,\ldots,n\}$ and
$\big\{f_i\big(b_1^{-1},b_2^{-1};t^{-1};z\big)\,|\,i=0,1,\ldots,n\big\}$ form bases of the linear space spanned by $\{m_\lambda(z)\,|\,\lambda<(1^n)\}$,
the polynomial $e_i(a_2,b_1;z)$ is expressed as a linear combination of $f_i\big(b_1^{-1},b_2^{-1};t^{-1};z\big)$ $(i=0,1,\ldots,n)$, i.e.,
\begin{equation*}
e_{n-j}(a_2,b_1;z)=\sum_{i=0}^{n}\tilde{l}^{\tinyR\,\prime}_{ij}
f_{n-i}\big(b_1^{-1},b_2^{-1};t^{-1};z\big),
\end{equation*}
where $\tilde{l}^{\tinyR\,\prime}_{ij}$ are some constants. From~\eqref{eq:vanishing-f} we have
\begin{equation}\label{eq:tilde l'2}
\tilde{l}^{\tinyR\,\prime}_{ij}=e_{n-j}\big(a_2,b_1;\xi_{n-i}\big(b_1^{-1},b_2^{-1};t^{-1}\big)\big)
=e_{n-j}\big(a_2,b_1;\zeta_{n-i}\big(x,b_1^{-1}\big)\big)\Big|_{x=b_2^{-1}t^{-(i-1)}}.
\end{equation}
Using \eqref{eq:tri-zeta(xb)2} in Lemma \ref{lem:tri-zeta(xb)} we have
\begin{gather*}
e_{n-j}\big(a_2,b_1;\zeta_{n-i}\big(x,b_1^{-1}\big)\big)=\big(xb_1t^{n-j};t\big)_j \big(xa_2^{-1};t\big)_{i-j}\big(a_2^{-1}b_1^{-1}t^{-(n-i-1)};t\big)_{n-i}\\
\hphantom{e_{n-j}\big(a_2,b_1;\zeta_{n-i}\big(x,b_1^{-1}\big)\big)=}{}\times
\frac{(t;t)_{n-j}(t;t)_{i}}{t^{{n\choose 2}}(1-t)^n(t;t)_{i-j}}.
\end{gather*}
Combining this and \eqref{eq:tilde l'2}, we obtain \eqref{eq:c-(e)=(f)tildeL'}.

On the other hand,
since $\{e_i(a_1,b_2;z)\,|\,i=0,1,\ldots,n\}$ is also a basis of the linear space spanned by $\{m_\lambda(z)\,|\,\lambda<(1^n)\}$,
the polynomial $f_i\big(b_1^{-1},b_2^{-1};t^{-1};z\big)$ is expressed as a linear combination of $e_i(a_1,b_2;z)$ $(i=0,1,\ldots,n)$, i.e.,
\begin{equation*}
f_{n-j}\big(b_1^{-1},b_2^{-1};t^{-1};z\big)
=\sum_{i=0}^{n}
\tilde{u}^{\tinyR\,\prime}_{ij}
 e_i(a_1,b_2;z),
\end{equation*}
where $\tilde{u}^{\tinyR\,\prime}_{ij}$ are some constants. From~\eqref{eq:matsuo2} we have
\begin{gather}
\tilde{u}^{\tinyR\,\prime}_{ij}
=\frac{f_{n-j}\big(b_1^{-1},b_2^{-1};t^{-1};\zeta_i\big(a_1,b_2^{-1}\big)\big)}{c_i}\nonumber\\
\hphantom{\tilde{u}^{\tinyR\,\prime}_{ij}}{}
=\frac{f_{n-j}\big(b_1^{-1},b_2^{-1};t^{-1};\xi_{n-i}\big(x,b_2^{-1};t^{-1}\big)\big)}{c_i}\Big|_{x=a_1t^{n-i-1}},\label{eq:tilde u'2}
\end{gather}
where $c_i$ is the constant given in \eqref{eq:c-i-a} as
\begin{equation}\label{eq:c-i-3}
c_i=\big(a_1b_2t^{i};t\big)_{n-i}\big(a_1^{-1}b_2^{-1}t^{-(i-1)};t\big)_{i}
\frac{(t;t)_i(t;t)_{n-i}}{t^{n\choose 2}(1-t)^n}.
\end{equation}
From \eqref{eq:f_i(xi_j(x,a_2;t))} in Corollary \ref{cor:Triangularity-f2} we have
\begin{gather*}
f_{n-j}\big(b_1^{-1},b_2^{-1};t^{-1};\xi_{n-i}\big(x,b_2^{-1};t^{-1}\big)\big)\\
\qquad{} =
\frac{\big(xb_1;t^{-1}\big)_{j-i}\big(xb_2t^i;t^{-1}\big)_{n-j}(t;t)_{n-i}}
{\big(b_1b_2^{-1}t^{-(i+j-n)};t^{-1}\big)_{j-i}\big(b_1^{-1}b_2t^j;t^{-1}\big)_{n-j}(t;t)_{j-i}(t;t)_{n-j}}.
\end{gather*}
Combining this, \eqref{eq:tilde u'2} and \eqref{eq:c-i-3}, we therefore obtain the expression~\eqref{eq:c-(f)=(e)tildeU'}.
\end{proof}

\begin{proof}[Proof of Theorem \ref{thm:R=LDU}]
From \eqref{eq:(e)=(f)tildeU} and \eqref{eq:(f)=(e)tildeL} in Lemma \ref{lem:(e)=(f)tildeU}, we have
\begin{gather}
\big(e_n(a_2,b_1;z),e_{n-1}(a_2,b_1;z),\ldots,e_0(a_2,b_1;z)\big)\nonumber\\
 \qquad{}=\big(e_0(a_1,b_2;z),e_1(a_1,b_2;z),\ldots,e_n(a_1,b_2;z)\big)\tilde{L}_\tinyR\tilde{U}_\tinyR,\label{eq:(e)=(e)tildeLU}
\end{gather}
where $\tilde{L}_\tinyR=\big(\tilde{l}^\tinyR_{ij}\big)_{0\le i,j\le n}$ and
$\tilde{U}_\tinyR=\big(\tilde{u}^\tinyR_{ij}\big)_{0\le i,j\le n}$
are the matrices given by~\eqref{eq:tilde u} and~\eqref{eq:tilde l}, respectively.
Comparing \eqref{eq:(e)=(e)tildeLU} with~\eqref{eq:R=LDU},
we obtain $R=\tilde{L}_\tinyR\tilde{U}_\tinyR=L_\tinyR D_\tinyR U_\tinyR$, i.e.,
\begin{equation*}
l^\tinyR_{ij}=\frac{\tilde{l}^\tinyR_{ij}}{\tilde{l}^\tinyR_{jj}}
,
\qquad
d^\tinyR_{j}=\tilde{l}^\tinyR_{jj}\tilde{u}^\tinyR_{jj},
\qquad
u^\tinyR_{ij}=\frac{\tilde{u}^\tinyR_{ij}}{\tilde{u}^\tinyR_{ii}}.
\end{equation*}
Lemma \ref{lem:(e)=(f)tildeU} implies that $l^\tinyR_{ij}$, $d^\tinyR_{j}$ and $u^\tinyR_{ij}$ above coincide with~\eqref{eq:lRij}, \eqref{eq:dRij} and~\eqref{eq:uRij}, respectively.
On the other hand, from~\eqref{eq:(e)=(f)tildeL'} and \eqref{eq:(f)=(e)tildeU'} in Lemma~\ref{lem:(e)=(f)tildeL'}, we have
\begin{gather}
\big(e_n(a_2,b_1;z),e_{n-1}(a_2,b_1;z),\ldots,e_0(a_2,b_1;z)\big)\nonumber\\
 \qquad{}=\big(e_0(a_1,b_2;z),e_1(a_1,b_2;z),\ldots,e_n(a_1,b_2;z)\big)
\tilde{U}'{}_{\!\!\!\tinyR}\tilde{L}'{}_{\!\!\tinyR},\label{eq:(e)=(e)tildeU'L'}
\end{gather}
where $\tilde{U}'{}_{\!\!\!\tinyR}=\big(\tilde{u}^{\tinyR\,\prime}_{ij}\big)_{0\le i,j\le n}$ and
$\tilde{L}'{}_{\!\!\tinyR}=\big(\tilde{l}^{\tinyR\,\prime}_{ij}\big)_{0\le i,j\le n}$
are the matrices given by~\eqref{eq:c-(e)=(f)tildeL'} and~\eqref{eq:c-(f)=(e)tildeU'}, respectively.
Comparing \eqref{eq:(e)=(e)tildeU'L'} with \eqref{eq:R=LDU}, we obtain
$R=\tilde{U}'{}_{\!\!\!\tinyR}\tilde{L}'{}_{\!\!\tinyR}
=U'{}_{\!\!\!\tinyR}D'{}_{\!\!\tinyR}L'{}_{\!\!\tinyR}$, i.e.,
\begin{equation*}
u^{\tinyR\,\prime}_{ij}=\frac{\tilde{u}^{\tinyR\,\prime}_{ij}}{\tilde{u}^{\tinyR\,\prime}_{jj}},
\qquad
d^{\tinyR\,\prime}_{j}=\tilde{u}^{\tinyR\,\prime}_{jj}\,\tilde{l}^{\tinyR\,\prime}_{jj},\qquad
l^{\tinyR\,\prime}_{ij}=\frac{\tilde{l}^{\tinyR\,\prime}_{ij}}{\tilde{l}^{\tinyR\,\prime}_{ii}}.
\end{equation*}
Lemma \ref{lem:(e)=(f)tildeL'} implies that
$u^{\tinyR\,\prime}_{ij}$,
$d^{\tinyR\,\prime}_{j}$ and
$l^{\tinyR\,\prime}_{ij}$ above coincide with~\eqref{eq:u'Rij}, \eqref{eq:d'Rij} and~\eqref{eq:l'Rij}, respectively.
\end{proof}

\appendix
\section{Appendix}\label{sectionA}

In this appendix we consider the Gauss decomposition part
$A=U'{}_{\!\!\!\tinyA}D'{}_{\!\!\tinyA}L'{}_{\!\!\tinyA}$ of Theorem~\ref{thm:main}.
Since the method to compute $A=U'{}_{\!\!\!\tinyA}D'{}_{\!\!\tinyA}L'{}_{\!\!\tinyA}$ is
almost the same as that to compute $A=L_\tinyA D_\tinyA U_\tinyA$,
we only give the outline of the proof. For this purpose, we define another
family of interpolation polynomials
$\tilde{E}'_{k,i}(a,b;z)$ slightly different from (\ref{eq:Eki2}).
Set
\begin{equation*}
\tilde{E}'_{k,i}(z)=\tilde{E}'_{k,i}(a,b;z):={\cal A}E'_{k,i}(a,b;z)/\Delta(z),
\end{equation*}
where
\begin{equation*}
E'_{k,i}(a,b;z):=\underbrace{z_{n-k+1}z_{n-k+2}\cdots z_n\phantom{\Big|}\!\!}_k
\Delta(t;z)
\prod_{j=1}^{n-i}(1-bz_j)\prod_{j=n-i+1}^n\big(1-a^{-1}z_j\big).
\end{equation*}

We now specify $a=a_1$, $b=b_2$, i.e., we set $\tilde{E}'_{k,i}(z)=\tilde{E}'_{k,i}(a_1,b_2;z)$ throughout this section.
\begin{Lemma}[three-term relations]
Suppose $k\le i$. Then,
\begin{gather}
a_2^{-1}\big(1-q^\alpha a_1a_2b_1b_2 t^{2n-k-1}\big)\big\la\tilde{E}'_{k,i}\big\ra\nonumber\\
\qquad{} =t^{k-1}\big(1-q^\alpha a_1b_1t^{2n-k-i}\big)\big\la\tilde{E}'_{k-1,i}\big\ra
- q^\alpha t^{n-1}\big(1-a_1b_1t^{n-i}\big)\big\la\tilde{E}'_{k-1,i-1}\big\ra.\label{eq:3term03}
\end{gather}
On the other hand, if $k\ge i$, then,
\begin{gather}
a_1^{-1}\big(1-q^\alpha a_1b_1t^{2n-k-i}\big)\big\la\tilde{E}'_{k,i-1}\big\ra\nonumber\\
 \qquad{} =
t^{k-i}\big(1-q^\alpha t^{n-k}\big)\big\la\tilde{E}'_{k-1,i-1}\big\ra
-a_2^{-1}t^{-(i-1)}\big(1-a_2b_2 t^{i-1}\big)\big\la\tilde{E}'_{k,i}\big\ra.\label{eq:3term04}
\end{gather}
\end{Lemma}

\begin{proof}Put $
\tilde{\varphi}'_{k,i-1}(z)
:={\cal A}\nabla_{\!1}\varphi'_{k,i-1}(z)
$, where
\begin{equation*}
\varphi'_{k,i-1}(z):=\big(1-a_1^{-1}z_1\big)\big(1-a_2^{-1}z_1\big)\prod_{j=2}^n(z_1-tz_j)
\times E_{k-1,i-1}^{(n-1)}(z_2,\ldots,z_n).
\end{equation*}
Then, by a similar argument to that
used in the proof of Lemma~\ref{lem:3term1st-01},
it follows that,
if $k\le i$, the polynomial $\tilde{\varphi}'_{k,i-1}(z)$ satisfies
\begin{equation}
\label{eq:3term03phi}
\frac{\tilde{\varphi}'_{k,i-1}(z)}{\Delta(z)}
=c'_{k,i}\tilde{E}'_{k,i}(z)+c'_{k-1,i}\tilde{E}'_{k-1,i}(z)+c'_{k-1,i-1}\tilde{E}'_{k-1,i-1}(z),
\end{equation}
where
\begin{gather*}
c'_{k,i}=-a_2^{-1}t^{n-1}\big(1-q^\alpha a_1a_2b_1b_2 t^{2n-k-1}\big),\\
c'_{k-1,i}=t^{n+k-2}\big(1-q^\alpha a_1b_1t^{2n-k-i}\big),\\
c'_{k-1,i-1}=-q^\alpha t^{2n-2}\big(1-a_1b_1t^{n-i}\big),
\end{gather*}
while, if $i\le k$, then $\tilde{\varphi}'_{k,i-1}(z)$ satisfies
\begin{equation}\label{eq:3term04phi}
\frac{\tilde{\varphi}'_{k,i-1}(z)}{\Delta(z)}=d'_{k,i}\tilde{E}'_{k,i}(z)+d'_{k,i-1}\tilde{E}'_{k,i-1}(z) +d'_{k-1,i-1}\tilde{E}'_{k-1,i-1}(z),
\end{equation}
where
\begin{gather*}
d'_{k,i}=-a_2^{-1}t^{n-1}\big(1-a_2b_2 t^{i-1}\big),\\
d'_{k,i-1}=-a_1^{-1}t^{n+i-2}\big(1-q^\alpha a_1b_1t^{2n-k-i}\big),\\
d'_{k-1,i-1}=t^{n+k-2}\big(1-q^\alpha t^{n-k}\big).
\end{gather*}
Using the expressions (\ref{eq:3term03phi}) and (\ref{eq:3term04phi})
for $\tilde{\varphi}_{k,i-1}(z)$,
(\ref{eq:3term03}) and~(\ref{eq:3term04})
follow by application of Lemma~\ref{lem:nabla=0}.
\end{proof}

By repeated use of the three-term relations~\eqref{eq:3term03} and~\eqref{eq:3term04},
we obtain the following.
\begin{Lemma}
If $k\le i$, then
\begin{equation*}
\big\la\tilde{E}'_{k,i}\big\ra=\sum_{j=0}^l U'^{k,i}_{k-l,i-j}\big\la\tilde{E}'_{k-l,i-j}\big\ra,
\end{equation*}
where
\begin{equation*}
U'^{k,i}_{k-l,i-j}
=\big({-}q^\alpha t^{n-k+l-1}\big)^j \big(a_2t^{k-l}\big)^lt^{{l-j \choose 2}}
\qbin{l}{j}{t}
\frac{\big(a_1b_1t^{n-i};t\big)_j\big(q^\alpha a_1b_1 t^{2n-k-i+j};t\big)_{l-j}
}{\big(q^\alpha a_1a_2b_1b_2 t^{2n-k-1};t\big)_l},
\end{equation*}
while, if $k\ge i$, then
\begin{equation*}
\big\la\tilde{E}'_{k,i}\big\ra=\sum_{j=0}^l L'^{k,i}_{k-l+j,i+j}\big\la\tilde{E}'_{k-l+j,i+j}\big\ra,
\end{equation*}
where
\begin{equation*}
L'^{k,i}_{k-l+j,i+j}
=
\qbin{l}{j}{t}
\frac{\big({-}a_2^{-1}t^{-(k-1)}\big)^j\big(a_1t^{k-i-1}\big)^lt^{-{l\choose 2}}\big(q^\alpha t^{n-k};t\big)_{l-j}\big(a_2b_2t^i;t\big)_j}
{\big(q^\alpha a_1b_1 t^{2n-k-i-j-1};t\big)_{l-j}\big(q^\alpha a_1b_1 t^{2n-k-i-2j+l};t\big)_j}.
\end{equation*}
\end{Lemma}

As a special case of the above lemma we immediately have the following.
\begin{Lemma}\label{lem:u'l'-A}
For $0\le j\le n$, $\big\la \tilde{E}'_{j,j}\big\ra$ is expressed as
\begin{equation}\label{eq:<E'j,j>}
\big\la\tilde{E}'_{j,j}\big\ra=\sum_{i=0}^j \tilde{u}'_{ij}\big\la\tilde{E}'_{0,i}\big\ra,
\end{equation}
where
\begin{equation}\label{eq:<E'j,j>c}
\tilde{u}'_{ij}=U'^{j,j}_{0,i}=\big({-}q^\alpha t^{n-1}\big)^{j-i}a_2^jt^{{i\choose 2}}
\qbin{j}{i}{t}
\frac{\big(a_1b_1t^{n-j};t\big)_{j-i}\big(q^\alpha a_1b_1 t^{2n-i-j};t\big)_i
}{\big(q^\alpha a_1a_2b_1b_2 t^{2n-j-1};t\big)_j},
\end{equation}
while, for $0\le j\le n$, $\big\la \tilde{E}'_{n,j}\big\ra$ is expressed as
\begin{equation}\label{eq:<E'n,j>}
\big\la\tilde{E}'_{n,j}\big\ra=\sum_{i=j}^n \tilde{l}'_{ij}\big\la\tilde{E}'_{i,i}\big\ra,
\end{equation}
where
\begin{equation}\label{eq:<E'n,j>c}
\tilde{l}'_{ij}=L'^{n,j}_{i,i}=
(-1)^{i-j}
\qbin{n-j}{n-i}{t}
\frac{a_1^{n-j}a_2^{-(i-j)}t^{{n-i\choose 2}+{j\choose 2}-{i\choose 2}}\big(q^\alpha;t\big)_{n-i}\big(a_2b_2t^j;t\big)_{i-j}}
{\big(q^\alpha a_1b_1 t^{n-i-1};t\big)_{n-i}\big(q^\alpha a_1b_1 t^{2(n-i)};t\big)_{i-j}}.
\end{equation}
\end{Lemma}

\begin{proof}[Proof of \eqref{eq:u'Aij}--\eqref{eq:l'Aij} in Theorem \ref{thm:main}]
From (\ref{eq:<E'j,j>}), we have
\begin{equation*}
\big(\big\la \tilde{E}'_{0,0}\big\ra,\big\la \tilde{E}'_{1,1}\big\ra,\ldots,\big\la \tilde{E}'_{n-1,n-1}\big\ra,\big\la \tilde{E}'_{n,n}\big\ra\big)
=\big(\big\la \tilde{E}'_{0,0}\big\ra,\big\la \tilde{E}'_{0,1}\big\ra,\ldots,\big\la \tilde{E}'_{0,n-1}\big\ra,\big\la \tilde{E}'_{0,n}\big\ra\big)\tilde{U}',
\end{equation*}
where the matrix $\tilde{U}'=\big(\tilde{u}'_{ij}\big)_{0\le i,j\le n}$ is defined by~(\ref{eq:<E'j,j>c}).
Moreover, from~(\ref{eq:<E'n,j>}) we have
\begin{align}
\big(\big\la \tilde{E}'_{n,0}\big\ra,\big\la \tilde{E}'_{n,1}\big\ra,\ldots,\big\la \tilde{E}'_{n,n-1}\big\ra, \big\la \tilde{E}'_{n,n}\big\ra\big)
&=\big(\big\la \tilde{E}'_{0,0}\big\ra,\big\la \tilde{E}'_{1,1}\big\ra,\ldots,\big\la \tilde{E}'_{n-1,n-1}\big\ra,\big\la \tilde{E}'_{n,n}\big\ra\big)\tilde{L}'
\nonumber\\
&=\big(\big\la \tilde{E}'_{0,0}\big\ra,\big\la \tilde{E}'_{0,1}\big\ra,\ldots,\big\la \tilde{E}'_{0,n-1}\big\ra,\big\la \tilde{E}'_{0,n}\big\ra\big)\tilde{U}'\tilde{L}',\label{eq:-U-L}
\end{align}
where the matrix $\tilde{L}'=\big(\tilde{l}'_{ij}\big)_{0\le i,j\le n}$ is defined by (\ref{eq:<E'n,j>c}).
Since $T_{\alpha}\Phi(z)=z_1z_2\cdots z_n\Phi(z)$ and $z_1z_2\cdots z_n \tilde{E}'_{0,i}(z)=\tilde{E}'_{n,i}(z)$, we have $T_{\alpha}\big\la \tilde{E}'_{0,i}\big\ra=\big\la \tilde{E}'_{n,i}\big\ra$, i.e.,
\begin{equation}\label{eq:T-a(E')}
T_{\alpha}\big(\big\la \tilde{E}'_{0,0}\big\ra,\big\la \tilde{E}'_{0,1}\big\ra,\ldots,\big\la \tilde{E}'_{0,n-1}\big\ra,\big\la \tilde{E}'_{0,n}\big\ra\big)
=\big(\big\la \tilde{E}'_{n,0}\big\ra,\big\la \tilde{E}'_{n,1}\big\ra,\ldots,\big\la \tilde{E}'_{n,n-1}\big\ra,\big\la \tilde{E}'_{n,n}\big\ra\big).
\end{equation}
From (\ref{eq:-U-L}) and (\ref{eq:T-a(E')}),
we obtain the difference system
\begin{equation*}
T_{\alpha}\big(\big\la \tilde{E}'_{0,0}\big\ra,\big\la \tilde{E}'_{0,1}\big\ra,\ldots,\big\la \tilde{E}'_{0,n-1}\big\ra,\big\la \tilde{E}'_{0,n}\big\ra\big)
=\big(\big\la \tilde{E}'_{0,0}\big\ra,\big\la \tilde{E}'_{0,1}\big\ra,\ldots,\big\la \tilde{E}'_{0,n-1}\big\ra,\big\la \tilde{E}'_{0,n}\big\ra\big)\tilde{U}'\tilde{L}'.
\end{equation*}
Comparing this with (\ref{eq:main}), we therefore obtain
$A=\tilde{U}'\tilde{L}'
=U'{}_{\!\!\!\tinyA}D'{}_{\!\!\tinyA}L'{}_{\!\!\tinyA},
$
i.e.,
\begin{equation*}
u^{\tinyA\,\prime}_{ij}=\frac{\tilde{u}'_{ij}}{\tilde{u}'_{jj}}
, \qquad
d^{\tinyA\,\prime}_{j}=\tilde{u}'_{jj}\tilde{l}'_{jj},
\qquad
l^{\tinyA\,\prime}_{ij}=\frac{\tilde{l}'_{ij}}{\tilde{l}'_{ii}}
.
\end{equation*}
Lemma \ref{lem:u'l'-A} implies that
$u^{\tinyA\,\prime}_{ij}$,
$d^{\tinyA\,\prime}_{j}$ and
$l^{\tinyA\,\prime}_{ij}$ above
coincide with \eqref{eq:u'Aij}, \eqref{eq:d'Aij} and \eqref{eq:l'Aij}, respectively, which completes the proof.
\end{proof}

Finally we give an explicit forms for ${L'_A}^{\!-1}$.

\begin{Proposition}\label{prop:inverseL'_A}
The inverse matrix $L'{}_{\!\!\tinyA}^{-1}=\big(l^{\tinyA\,\prime *}_{ij}\big)_{0\le i,j\le n}$ is lower triangular and is written as
\begin{equation*}%\label{eq:inverseL'_A}
l^{\tinyA\,\prime *}_{ij}
=
\qbin{n-j}{n-i}{t}
\frac{\big(a_1a_2^{-1}t^{-j}\big)^{i-j}\big(a_2b_2t^j;t\big)_{i-j}}
{\big(q^\alpha a_1b_1 t^{2n-i-j-1};t\big)_{i-j}}.
\end{equation*}
\end{Proposition}

\begin{proof}Using \eqref{eq:3term04} we can calculate the entries
$l^{\tinyA\,\prime *}_{ij}$ of the lower triangular matrix $L'{}_{\!\!\tinyA}^{-1}$
by completely the same way as Proposition~\ref{prop:inverseU_A}. We omit the details.
\end{proof}

\subsection*{Acknowledgements}

The author would like to thank the anonymous referees
who kindly provided many careful comments and suggestions for
improving his manuscript. He would also like to express his gratitude to
Professor Peter J.~Forrester for providing considerable encouragement from the early stage of this research.
The comments and suggestions by Professor Yasuhiko Yamada
on the preliminary version of this manuscript are most appreciated.
This work was supported by JSPS KAKENHI Grant Number (C)18K03339.

\pdfbookmark[1]{References}{ref}
\LastPageEnding

\end{document}